\documentclass{amsart}

\usepackage{geometry}
\usepackage{lipsum}
\usepackage{amsfonts}
\usepackage{graphicx}
\usepackage{epstopdf}
\usepackage{pifont}
\usepackage{enumitem}
\usepackage{tikz}
\usepackage{pgfplots}
\usepackage{pgfplotstable}
\usetikzlibrary{calc}
\usepackage{tkz-base}
\usepackage{tkz-euclide}
\pgfplotsset{compat=1.17}

\usepackage{hyperref} 
\usepackage{xcolor}
\hypersetup{
    colorlinks=true,
    linkcolor=blue,  
    urlcolor=cyan,
    citecolor=green!70!black,
    pdftitle={},
    pdfauthor={},
}

\usepackage{amsopn}
\usepackage{siunitx}
\def\colmet{gray!30!white} 
\def\colox{orange!40!white}
\def\colsol{blue!20!white}

\newcommand\dd{{\operatorname{d}}}

\newcommand\dt{\Delta t}
\newcommand\Fluxnum[1]{\mathcal{F}_{#1 \frac{1}{2}}^n}
\newcommand\h[1]{h_{#1 \frac{1}{2}}}
\renewcommand\v[2]{v_{#1 \frac{1}{2}}^{#2}}
\newcommand\Bern[3]{B\left(#3 L^n \h{#1}{} \v{#1}{#2}\right)}
\newcommand\x[2]{x_{#1 \frac{1}{2}}^{#2}}
\newcommand\R{\mathbb{R}}
\newcommand\Flux{\mathcal{F}}
\renewcommand\d[1]{\partial_{#1}}
\renewcommand\H{\mathcal{H}_{\phi}}
\newcommand\intmov{\int_{X_0(t)}^{X_1(t)}}

\newcommand\normeh[1]{{\Vert {#1} \Vert}_1}
\newcommand\normelh[1]{{\Vert {#1} \Vert}_{2,1}}
\newcommand\normehs[1]{{\Vert {#1} \Vert}_{-1}}
\newcommand\normelhs[1]{{\Vert {#1} \Vert}_{2,-1}}
\newcommand\ddisc{\partial_{t,\dt}}

\newcommand\Mm{{\mathcal M}}
\newcommand\Ff{{\mathcal F}}
\newcommand\bz{{\boldsymbol z}}
\newcommand\bu{{\boldsymbol u}}
\newcommand\bX{{\boldsymbol X}}
\newcommand\p{\partial}

\newcounter{cst}
\newcommand{\ctel}[1]{C_{\refstepcounter{cst}\label{#1}\thecst}}
\newcommand{\cter}[1]{C_{\ref{#1}}}

\newtheorem{theorem}{Theorem}[section]
\newtheorem{remark}[theorem]{Remark}

\newtheorem{lemma}[theorem]{Lemma}
\newtheorem{corollary}[theorem]{Corollary}
\newtheorem{proposition}[theorem]{Proposition}

\newtheorem{definition}{Definition}

\title{Finite Volumes for a dissipative free boundary problem}
\author{Cl\'ement Canc\`es}
\author{Claire Chainais-Hillairet}
\author{Am\'elie Dupouy}
\address{Univ. Lille, CNRS, Inria, UMR 8524 - Laboratoire Paul Painlevé, F-59000 Lille, France}
 \email{clement.cances@inria.fr, claire.chainais@univ-lille.fr, amelie.dupouy@inria.fr}

\begin{document}

\begin{abstract}
We study a toy model for the evolution of the oxygen concentration in an oxide layer. 
It consists in a transient convection diffusion equation in a one-dimensional domain of variable width. 
The motions of the boundaries are governed by the traces of the concentration. 
We exhibit a necessary and sufficient condition on the parameters involved in the model for the existence of a unique traveling-wave solution. 
Moreover, we show that the model admits some universal entropy structure, in the sense that any convex function of the concentration
 yields a dissipated free energy (up to exchanges with the outer environment at the boundaries).  
We propose then an implicit in time arbitrary Lagrangian-Eulerian finite volume scheme based on Scharfetter-Gummel fluxes. 
It is shown to be unconditionally convergent, to preserve exactly the travelling wave, and to dissipate all the aforementioned free energies. 
Numerical experiments show that our scheme is first order accurate in time and second order in space, and that the transient solution 
converges in the long-time limit towards the traveling-wave solution. 
\end{abstract}



\maketitle

\section{Introduction}

\subsection{Motivation and context}

The so-called  \textit{Diffusion  Poisson Coupled Model} (DPCM) introduced in \cite{Bataillon_elec} aims at understanding finely the dynamics of an oxide (magnetite) layer at the 
surface of iron canisters to be stored in deep underground nuclear waste repositories. Because of its complexity --3 different charged species evolving in a free boundary domain with retroaction of a self-consistent electric potential-- and of its lack of mathematical structure --no Lyapunov functional could be identified so-far--, no satisfying well-posedness was obtained so far on the full model. Moreover, numerical experiments presented in~\cite{Bataillon_elec, Calipso} show that, at least for particular choices of the parameters, the transient system converges in the long-time limit towards a travelling wave profile. 

Global in time existence results in the simpler case where the boundaries are fixed are provided in~\cite{Chainais_dcds, Cances_ZAMP}, while the existence of a travelling wave profile is shown in \cite{ChainGall} assuming the electroneutrality of the oxide. The existence of a travelling wave accounting for the coupling with a self-consistent electric field is also rigorously established in~\cite{Breden} thanks to computer assisted proof technics, but only for some particular set of parameters. 

In this paper we study a simplified model with only one chemical species and where electrical interactions are neglected. The goal there is to take advantage of the apparent simplicity of the model to answer some questions that were out of reach in more complex models. Is the existence of a travelling wave profile automatic, or are there some required conditions on the parameters to ensure the existence of such a particular solution? Can we extract a Lyapunov functional providing enough control on the problem so that we can rigorously prove the existence of a solution to the system? Does the travelling wave profile act as an attractor in the long-time limit? 
Can one design a numerical scheme that preserves at the discrete level the key features of the continuous problem? And can one rigorously establish the convergence of the scheme?

We show in what follows that, even in our simplified setting, conditions are required on the parameters to make the existence of a travelling wave profile possible. 
In our greatly simplified context, we even get a necessary and sufficient condition for the existence of the travelling wave profile, for which we derive an explicit expression. 

Besides, the problem under study appears to have an infinite number of free energies, which all dissipate up to some exchange terms with the outer environment. 
Any evolution of the solution profile or of the interface location generate an entropy production, and even more: the evolution is driven by the dissipation of the (infinite family of) free energies. In mathematical words, the system we consider here is a generalized gradient flow, in the sense of~\cite{Mie11}, and more precisely a port-gradient system~\cite{Mielke} since there are exchanges between the oxide and its surrounding environment. 
From a thermodynamical perspective, the fluxes are driven by the variations of the so-called chemical potentials as in \cite{Cances_ZAMP}, whereas the motions of the interfaces are governed by the 
pressure (or equivalently by the grand potential), in good accordance with the literature~\cite{PV99, CJ05, LLRV09}. Excepted for a very particular choice of the parameters, the system cannot reach an equilibrium, in opposition to the closely related configurations studied  in~\cite{PP10} or~\cite{CCCE_IFB}. As the numerical evidences show, the system rather converges towards a travelling wave profile. 
Such travelling waves for convection diffusion problems are also of great interest in other contexts, like for instance in models for cell motility~\cite{RB23, AM_arXiv}. 

In order to approximate the solutions to the continuous problem, we introduce an implicit in time arbitrary Lagrangian Eulerian (ALE) finite volume scheme (see for instance~\cite{ALE}). In our one-dimensional context, this approach is equivalent to the one proposed in~\cite{Calipso, CMZ18}, where a change of variable is introduced to formulate the problem on the fixed spatial domain $(0,1)$. 
Even after this change of variable, the continuous problem remains of convection diffusion type. We employ Scharfetter-Gummel type fluxes~\cite{ScharfGumm, Il'in69, Chatard_FVCA6}. These fluxes exactly capture the travelling wave, and dissipate all the convex free energies of the system. Building on these strong properties, we 
establish the convergence of the scheme towards a weak solution thanks to compactness arguments. 

We highlight in numerical experiments that the scheme is second order accurate in space and first order in time as expected. Furthermore, we observe that, in the long-time asymptotic regime, the solution to the problem converges exponentially fast towards the travelling wave profile (provided that the latter exists). Unfortunately, we have been unable to rigorously prove this result so far. Indeed, the port-gradient nature of the problem makes usual approaches such as~\cite{GajewGrog, GlitzGrogHunl, AMTU01, Glitz1, Glitz2, BC14} fail in our case. 
This will be the purpose of future work.

\subsection{Presentation of the continuous model}

In the simplified model studied in this paper, one chemical species --that should be thought as oxygen in the context of corrosion-- evolves inside the one-dimensional oxide layer represented by the interval $(X_0(t), X_1(t))$, see Figure~\ref{fig:geometry}. The oxide is surrounded by two different outer environments, modeling 
an aqueous solution for $x < X_0(t)$ and a metal subject to oxidation for $x>X_1(t)$.
Oxygen, the concentration of which being denoted by $u$ in the oxide, can enter the oxide layer from the solution, diffuse within the oxide, and 
then transform the metal into oxide thanks to an oxidation process. 

\begin{figure}[htb]
\centering
\begin{tikzpicture}[scale = 1]
\draw[fill = \colsol, color = \colsol, line width = 1pt] 
(0,0) -- (0,1) -- (3,1) -- (3,0) -- cycle;
\draw[fill = \colox, color = \colox, line width = 1pt] 
(3,0) -- (3,1) -- (5,1) -- (5,0) -- cycle;
\draw[fill = \colmet, color = \colmet, line width = 1pt] 
(5,0) -- (5,1) -- (7,1) -- (7,0) -- cycle;

\draw[color = red, line width = 1pt](3,0) -- (3,1);
\draw[color = red, line width = 1pt](5,0) -- (5,1);
\draw (1.5,.5) node{Solution};
\draw (6,.5) node{Metal};
\draw (4,.5) node{Oxide};

\draw[>=latex, <-] (3,1.5) -- (3.8, 1.9);
\draw[>=latex, <-] (5,1.5) -- (4.2, 1.9);
\draw (4,1.9) node[above]{Mobile interfaces};
\draw (3,1) node[above]{$\color{black}x=X_0(t)$};
\draw (5,1) node[above]{$\color{black}x=X_1(t)$};
\end{tikzpicture}
\caption{Schematic representation of the oxide in its environment.}\label{fig:geometry}
\end{figure}
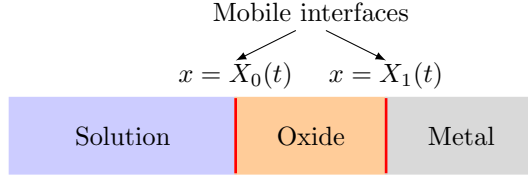

From a mathematical perspective, the problem amounts to find $(u, X_0, X_1)$, where $X_\ell : [0,T) \rightarrow \R$ for $\ell \in \{0,1\}$, and
\[
u: Q_T = \bigcup_{t \in [0,T)} \left( \{t\} \times [X_0(t), X_1(t)] \right)  \to \R_{\geq 0}
\]
with $T>0$, 
such that, for any $(t,x)$ in $Q_T$, it holds
\begin{subequations} \label{eq:sys_mov}
\begin{align}
\d{t} u(t,x) + \d{x} J(t,x) & = 0,\label{eq:u} \\
J(t,x) &= - \d{x} u(t,x) + (1-R) X'_1(t) u(t,x),  \label{eq:def_J}\\
J(t, X_0(t)) - u(t, X_0(t))X_0'(t)  & = a - b\, u(t,X_0(t)),  \label{eq:flux_bord0}\\
J(t, X_1(t)) - u(t, X_1(t))X_1'(t) &= 0, \label{eq:flux_bord1}\\
X'_0(t) &= \alpha_0 - \beta_0\, u(t,X_0(t)) + (1-R) X'_1(t), \label{eq:X0_cont}\\
X'_1(t) & = - \alpha_1 + \beta_1\, u(t,X_1(t)), \label{eq:X1_cont}\\
u(0,x) & = u^\text{init}(x), \label{eq:CI} \\
X_0(0) & = 0, \label{eq:CI.X0} \\
X_1(0) & = L^0, \label{eq:CI.X1} 
\end{align}
\end{subequations}

The first equation \eqref{eq:u} is the conservation law on the concentration of oxygen inside the oxide. The flux $J$ in \eqref{eq:def_J} is composed of two terms: a diffusion flux, and a transport term related to the movement of the oxide with respect to the reference frame attached to the metal. The frame of the oxide moves in this reference frame with velocity 
\[
v_\text{ox}(t) = (1 - R) X'_1(t),
\] where $R$ is the so-called Pilling-Bedworth ratio encoding the difference in molar volumes between the metal and oxide. $R$ is assumed to be a given positive constant.

The conservative fluxes $\mathcal{F}_\ell$ at the interfaces $X_\ell$ are defined by 
	\begin{equation} \label{eq:rel_flux}
		\Flux_\ell(t) = J(t, X_\ell(t)) - u(t, X_\ell(t))X'_\ell(t) \quad \textnormal{for} \quad \ell \in \{0,1\}.
	\end{equation}
There is no oxygen exchange between the oxide and the metal, which transcripts through the condition on the conservative fluxes in \eqref{eq:flux_bord1}. The other condition \eqref{eq:flux_bord0} is obtained thanks to the mass action law, that becomes in this case a Robin condition at the interface between the solution and the oxide. In~\eqref{eq:flux_bord0}, the kinetic parameters $a$ and $b$ are given and positive.

Both \eqref{eq:X0_cont} and \eqref{eq:X1_cont} describe the evolution of the oxide boundaries along time. They involve the positive given kinetic parameters $\alpha_0, \beta_0, \alpha_1$, and $\beta_1$. The growth of the oxide over the metal is given by \eqref{eq:X1_cont}, which comes from the mass action law, and \eqref{eq:X0_cont} combines both the oxide frame velocity $v_\text{ox}$, and the dissolution velocity, denoted $v_\text{d}$, i.e. 
\[
X_0'(t) = v_\text{ox}(t) + v_\text{d}(t), \quad \text{with} \quad
v_\text{d}(t) = \alpha_0 - \beta_0 u(t, X_0(t)).
\]
Throughout the paper, we denote by $L(t) = X_1(t) - X_0(t) >0$ the thickness of the oxide. 
Since the problem is translation invariant, the initial value of the solution-oxide layer interface is set to 0 for simplicity in~\eqref{eq:CI.X0}. 
The initial position of the interface $X_1$ prescribed by~\eqref{eq:CI.X1} then directly encodes the initial thickness $L^0>0$ of the oxide layer. 
Finally, the initial concentration profile $u^\text{init}$ appearing in~\eqref{eq:CI} is assumed to be positive and bounded.

\subsection{Presentation of the numerical scheme}

Before defining our numerical scheme, let us first introduce some notation. We set a finite time horizon $T>0$, the value of which will be discussed later on. 
We also introduce a mesh of $[0,1]$, composed of $I > 0$ cells $(\xi_{i-\frac{1}{2}}, \xi_{i+\frac{1}{2}})$ such that
	\begin{equation*}
		0 = \xi_{\frac{1}{2}} < \xi_{\frac{3}{2}} < ... < \xi_{I- \frac{1}{2}} < \xi_{I+ \frac{1}{2}} = 1.
	\end{equation*}
	The size of each cell is denoted $h_i =\xi_{i+\frac{1}{2}} - \xi_{i-\frac{1}{2}}$ for $1 \leq i \leq I$, and the cell centers $\xi_i$, $0 \leq i \leq I$, are defined as 
	\[
	\xi_i = \frac{\xi_{i-\frac{1}{2}} + \xi_{i+\frac{1}{2}}}{2} \textnormal{ for } 1 \leq i \leq I, \qquad \xi_0 = \xi_{\frac{1}{2}}, \quad \text{and}\quad \xi_{I+1} = \xi_{I+ \frac{1}{2}}.
	\]
The distance between two consecutive cell centers $\xi_i$ is denoted by  $\h{i+}{} = \xi_{i+1} - \xi_i$ for $0 \leq i \leq I+1$. 
The mesh size $h$ is then defined as $h=\max_{1 \leq i \leq I} h_i$.

For notation convenience, we consider uniform time discretization. 
The time step is denoted $\dt$, and it is assumed that there exists an integer $N$ such that $N \dt = T$. The sequence $(t_n)_{0 \leq n \leq N}$ is defined as $t_n = n \dt$ for each $0 \leq n \leq N$. 
	
	Based on the $(0,1)$ mesh, at each set time $t > 0$, a mesh of $(X_0(t), X_1(t))$ is defined, with
	\(
		\x{i+}{}(t) = X_0(t) + L(t) \xi_{i+\frac{1}{2}}.
	\)

	It is then possible to look at the evolution of the oxygen concentration in a mesh cell $[\x{i-}{}(t),\x{i+}{}(t)]$, with $1 \leq i \leq I$. Using \eqref{eq:u}, it comes that
	\begin{equation} \label{eq:cons_mass}
		\frac{\dd}{\dd t} \int_{\x{i-}{}(t)}^{\x{i+}{}(t)} u(t,x) \dd x = - \mathcal{F}_{i+\frac{1}{2}}(t) + \mathcal{F}_{i-\frac{1}{2}}(t),
	\end{equation}
	where
	$\mathcal{F}_{i+\frac{1}{2}}(t) = J(t, \x{i+}{}(t)) - u(\x{i+}{}(t))\x{i+}{'}(t).$
	Bearing in mind the expression~\eqref{eq:def_J} of the flux $J$, we obtain
	\begin{equation*}
		\mathcal{F}_{i+\frac{1}{2}}(t) = - \d{x}u(t, \x{i+}{}(t)) + u(t, \x{i+}{}(t))\v{i+}{}(t), 
	\end{equation*}
	with the velocity $\v{i+}{}$ being defined by 
	\begin{equation}\label{eq:vi+1/2.cont}
		\v{i+}{}(t) = (1- R) X_1'(t) - L'(t)\xi_{i+\frac{1}{2}} - X_0'(t).
	\end{equation}
	
The initial conditions~\eqref{eq:CI}--\eqref{eq:CI.X1} are discretized into 
\begin{equation}\label{eq:u0.disc}
X_0^0 = 0, \quad X_1^0 = L^0, \quad\text{and}\quad  u_{i}^0 = \frac{1}{L^0 h_i}\int_{L^0\xi_{i-\frac12}}^{L^0\xi_{i+\frac12}} u^\text{init}(x) \dd x.
\end{equation} 

	A finite volume method is used to approximate the left-hand side of \eqref{eq:cons_mass}, and an implicit Euler method is used for the time discretization. 
	This leads to the following numerical scheme:
\begin{subequations} \label{eq:scheme}
\begin{align}
	\label{eq:scheme_u}
		h_i \frac{L^n u_i^n - L^{n-1} u_i^{n-1}}{\dt} + \Fluxnum{i+} - \Fluxnum{i-} &= 0, \quad \textnormal{ for } 1 \leq i \leq I, \\
	\label{eq:left_cond_flux}
		\Fluxnum{} - a + bu_0^n &= 0, \\
	\label{eq:right_cond_flux}
		\Fluxnum{I+} &= 0, \\
	\label{eq:X0}
	\frac{X_0^n - X_0^{n-1}}{\dt} - \alpha_0 + \beta_0 u_0^{n} - (1-R) \frac{X_1^n - X_1^{n-1}}{\dt} &= 0, \\
	\label{eq:X1}
		\frac{X_1^n - X_1^{n-1}}{\dt} + \alpha_1 - \beta_1 u_{I+1}^{n} &= 0, \\
		L^n - X_1^n + X_0^n &= 0.\label{eq:L}
\end{align}
\end{subequations}
As already mentioned in the introduction, we use Scharfetter-Gummel fluxes \cite{ScharfGumm}, that involve the 
so-called Bernoulli function defined by 
\(B(r) = \dfrac{r}{e^r-1}\), with continuous extension \(B(0) = 1\), in their definition:
	\begin{equation} \label{eq:def_flux_num}
		\Fluxnum{i+} = \frac{1}{L^n\h{i+}} \left( \Bern{i+}{n}{-} u_i^n - \Bern{i+}{n}{} u_{i+1}^n \right).
	\end{equation}
	In the above formula, the discrete velocities $\v{i+}{n}$ are natural discrete counterparts of the semi-discrete quantities~\eqref{eq:vi+1/2.cont}. Introducing the notation $\delta^n[f]=\frac{f^n-f^{n-1}}{\dt}$ for $f=X_0, X_1, L$, they write  
	\begin{equation}
	    \label{eq:def_v}
		\v{i+}{n} = (1-R) \delta^n [X_1] - \xi_{i+\frac{1}{2}} \delta^n [L] - \delta^n [X_0].
	\end{equation} 
	
	The choice of the specific expression~\eqref{eq:def_flux_num} for the numerical fluxes is motivated mainly by the fact that it guarantees the preservation of steady-state profiles such as the travelling wave in the case of this paper (see Section \ref{sec:tw}), while reaching the optimal second order accuracy in space~\cite{LMV96}.  More specifically, the Bernoulli function has some computational properties listed below that will be useful in the rest of the paper. 
	\begin{lemma} \label{lem:Bern}
		The Bernoulli function is $1$-Lipschitz continuous, positive, decreasing and satisfies:
		\begin{enumerate}[label=(\textit{\roman*})]
			\item For $r \in \R$, $B(-r) - B(r) = r$;
			\item For $r \in \R$, $B(r)e^r = B(-r)$;
			\item For $\theta \in [0,1]$, $p,q,r \in \R$:
				\begin{equation} \label{eq:Bern}
					B(-r)p - B(r)q = (\theta B(r) + (1-\theta) B(-r))(p-q) + r((1-\theta)q + \theta p).
				\end{equation}
		\end{enumerate}
	\end{lemma}

\subsection{Outline of the paper}

The aim of this paper is to conduct both a theoretical and a numerical analysis of the simplified corrosion model \eqref{eq:sys_mov}. One main goal is to prove the existence of a solution to the system \eqref{eq:sys_mov}, which is obtained thanks to the numerical analysis of the scheme \eqref{eq:scheme}. 
	
In a first part, the existence and the uniqueness of a travelling wave profile, as presented in \cite{ChainGall}, is established under some necessary and sufficient condition on the parameters of the model. Then, a result of decreasing free energy is established in Proposition \ref{prop:energy_dec}, and leads to several bounds on any solution of \eqref{eq:sys_mov} (see Proposition \ref{prop:u_bounds}).
	
Then, the analysis of the numerical scheme \eqref{eq:scheme} is conducted, beginning with the establishment of Proposition \ref{prop:disc_energ_dec}, which is a discrete counterpart of Proposition \ref{prop:energy_dec}. One can then prove the same bounds as for the continuous case with an analogous reasonning. This, combined with bounds on the domain size (see Proposition \ref{prop:Ln_bounds_non_unif}) leads to the existence of a solution to the numerical scheme \eqref{eq:sys_mov} (Theorem \ref{th:exist_scheme}). Some numerical experiments are provided in Section \ref{sec:simu} to illustrate the behavior of an approximate solution on several test cases. The numerical simulations also illustrate the long-time behaviour of the approximate solution, that numerically converges to a travelling wave profile, recovering the results stated in \cite{Bataillon_elec} and \cite{Calipso}.
	
The main result of the paper is the convergence of the solution of the scheme \eqref{eq:scheme} towards a weak solution of the system \eqref{eq:sys_mov} (Theorem \ref{th:conv}), and it is the goal of Section \ref{sec:conv}.
\section{Some properties of the continuous model}

\indent

In this section, the goal is to study the continuous model \eqref{eq:sys_mov}. First, the notion of travelling wave is defined, and its existence is proved in Proposition \ref{th:trav_wave}. In a second part, the non-decreasing behaviour of free energies associated to the problem (see Proposition \ref{prop:energy_dec}) leads to various energy estimates, particularly $L^\infty$ bounds, that are presented in Proposition \ref{prop:u_bounds}.
	\subsection{Travelling wave profile}
	The existence of a particular solution, that can be seen as a non-equilibrium steady state (as in \cite{Mielke} for example), is presented in this subsection. The existence of this particular profile, called a travelling wave profile, has been studied in \cite{Breden} for the general DPCM (\cite{Bataillon_elec}) . In the case of the model \eqref{eq:sys_mov}, the notion of travelling wave is defined as follows.
	\begin{definition}
		A triplet $(u, X_0, X_1)$ is called a travelling wave solution to \eqref{eq:sys_mov} if there exist a profile $\hat{u} : \R \rightarrow \R_{\geq 0}$, a velocity $\hat{c} \in \R$ and a width $\hat{L} \in \R_{>0}$ such that for all $t \in \R_{\geq 0}$
		\begin{equation} \label{eq:tw_bord}
			X_0'(t) = X_1'(t) = \hat{c} \textnormal{ and } X_1(t) - X_0(t) = \hat{L},
		\end{equation}
		and for all $(t,x) \in \R_{\geq 0} \times \bigcup_{t \in \R_{\geq 0}}(X_0(t), X_1(t))$
		\begin{equation} \label{eq:tw_u}
			u(t,x) = \hat{u}(x-\hat{c}t).
		\end{equation}
	\end{definition} 
	\begin{proposition} \label{th:trav_wave}
		Assume that the set of parameters satisfies either
		\begin{equation} \label{hyp} 
						\frac{\alpha_0 + R \alpha_1}{\beta_0 + R \beta_1} < \frac{a}{b} < \frac{\alpha_0}{\beta_0},
		\end{equation}
		or 
		\begin{equation} \label{hyp_inv} 
					\frac{\alpha_0 + R \alpha_1}{\beta_0 + R \beta_1} > \frac{a}{b} > \frac{\alpha_0}{\beta_0}.
		\end{equation}
		Then, there exists a unique travelling wave solution to \eqref{eq:sys_mov}, explicitly given by 
		\begin{subequations}\label{eq:tw..}
		\begin{align}
		\label{eq:sol_prog}
			\hat{u}(y) &= \frac{a}{b}e^{-R\hat{c}y } \: \textit{ for } y \in [0, \hat{L}], \\
		\label{eq:tw_velocity}
			\hat{c} & = \dfrac{1}{R}\left(\alpha_0 - \beta_0 \dfrac{a}{b}\right), \\
		\label{eq:tw_size}
			\hat{L} &= - \dfrac{1}{R\hat{c}} \log \left( \dfrac{b}{\beta_1 a} \left( \alpha_1 +\hat{c} \right) \right). 
		\end{align}
		\end{subequations}
		Under assumption 
		\begin{equation}\label{eq:equilibrium}
		 \frac{a}{b} = \frac{\alpha_0}{\beta_0}= \frac{\alpha_1}{\beta_1}, 
		 \end{equation}
		 there exists an infinite number of constant steady solutions with $\hat c = 0$, $\hat u(y) = \frac ab$ for all $y \in [0,\hat L]$ with $\hat L>0$ being arbitrary. 
		 If neither~\eqref{hyp}, \eqref{hyp_inv} nor \eqref{eq:equilibrium} hold true, then there exists no travelling wave solution. 
	\end{proposition}
	\begin{proof}
		Suppose that there exists a solution $(u, X_0, X_1)$ that satisfies \eqref{eq:tw_bord}-\eqref{eq:tw_u}. 		The equation \eqref{eq:u} can then be rewritten in terms of $\hat{u}$ as
		\begin{equation*} 
			- \d{y}^2 \hat{u} - R\hat{c} \d{y} \hat{u}=\d{y}\hat{\Flux} = 0, 
		\end{equation*}
		where we have set $\hat{\Flux} = - \d{y}\hat{u} - R\hat{c}\,\hat{u}$.
		Since the boundary condition \eqref{eq:flux_bord1} at $X_1$ gives $\hat{\Flux}(X_1(t)) = 0$ for any $t$, it leads to $\hat{\Flux} = 0$.
		We further deduce from the other boundary condition \eqref{eq:flux_bord0} that 
\(
			\hat{u}(0)=\dfrac{a}{b}.
\)
		Hence, the travelling wave profile is given by \eqref{eq:sol_prog}. 
		
		Plugging this expression in \eqref{eq:X0_cont} gives exactly the relation \eqref{eq:tw_velocity} for $\hat{c}$. We note that the sign of $\hat{c}$ coincides with the sign of $ \dfrac{\alpha_0}{\beta_0}-\dfrac{a}{b}$.
		
		Assume first that $\hat{c} = 0$, or equivalently that 
		$\dfrac{a}b = \dfrac{\alpha_0}{\beta_0}$, then we deduce from~\eqref{eq:sol_prog} that $\hat{u}(y)= \dfrac{a}{b}$ for all $y \in [0,\hat L]$. This is compatible with 
		 \eqref{eq:X1_cont} if and only if~\eqref{eq:equilibrium} holds.
		 Under this latter condition, the size of the domain $\hat{L}$ can be chosen arbitrarily as long as it remains positive. 
		 In other words, there exists an infinite number of constant equilibria under condition~\eqref{eq:equilibrium}. 

		Assume now that $\hat{c} \neq 0$. 
		We infer from~\eqref{eq:X1_cont} that 
		\(
		\dfrac{b(\alpha_1 + \hat c)}{a \, \beta_1} = e^{- R \hat c \hat L}. 
		\)
		This relation can be inverted into~\eqref{eq:tw_size} if and only if the left-hand side is positive, i.e. if $\alpha_1 + \hat c>0$. 
		This property is automatically satisfied if $\hat c>0$, i.e. when $\dfrac{\alpha_0}{\beta_0} > \dfrac{a}{b}$. But formula~\eqref{eq:tw_size} gives a positive thickness $\hat L$ if and only if 
		\[
		\frac{b(\alpha_1 + \hat c)}{\beta_1 a} < 1 \quad \Leftrightarrow \quad \frac{\alpha_0 + R \alpha_1}{\beta_0 + R \beta_1} < \frac{a}{b}.
		\]
		Similarly, in the case $\hat c<0$, i.e. when $\dfrac{\alpha_0}{\beta_0} < \dfrac{a}{b}$, the thickness $\hat L$ provided 
		by formula~\eqref{eq:tw_size} is positive if and only if 
		\[
		\frac{b(\alpha_1 + \hat c)}{\beta_1 a} > 1 \quad \Leftrightarrow \quad \frac{\alpha_0 + R \alpha_1}{\beta_0 + R \beta_1} > \frac{a}{b}.
		\]
		This last condition is in this case more stringent than $\alpha_1 + \hat c>0$, so that we end up with the simple conditions \eqref{hyp} or \eqref{hyp_inv} on the parameters. 

		To conclude this proof, one readily check that, when \eqref{hyp} or \eqref{hyp_inv} hold true, then $u(t,x) = \hat u(x - \hat c t)$, $X_0(t)= \hat c t $ 
		and $X_1(t) = \hat L + \hat c t$ with $\hat u, \hat c$ and $\hat L$  given by~\eqref{eq:tw..} is a solution to the continuous problem~\eqref{eq:sys_mov}. 
		\end{proof}
	\begin{remark}
		In the rest of the paper, the study is carried out under assumption \eqref{hyp}, which corresponds to the physical configuration encountered in the 
		corrosion process, where metal is transformed into oxide, which itself dissolves in the solution.
Moreover, remark that the assumption \eqref{hyp} immediately implies that $\dfrac{\alpha_1}{\beta_1} < \dfrac{\alpha_0}{\beta_0}$.
	\end{remark}
	\subsection{Energy estimates}
	In order to obtain some a priori estimates on the function $u$ solution to \eqref{eq:sys_mov}, we study the evolution in time of some free energies. In our case, any convex function can generate a free energy for the system. But, as the system is not isolated, in the sense that 
	there are exchanges of mass and of volume with the outer environments, the free energy is not decreasing along time. To circumvent this difficulty, we 
	follow the path proposed in~\cite{Cances_ZAMP} as 
	we introduce in the following definition what we refer to as the total free energy, encompassing contributions from the outer environments. 
	\begin{definition}
	Let $\phi \in \mathcal{C}^2(\R)$ be an arbitrary convex function. The free energy $\H$ associated to $\phi$ corresponding to the state $(u(t), X_0(t), X_1(t))$ is defined by
	\begin{equation}\label{eq:H}
	\H(t) = \intmov \phi(u(t,x)) {\text{d}x}, \qquad t \geq 0. 
	\end{equation}
	Moreover, defining the pressure function $\pi_\phi$ by $\pi_\phi(r) = r \phi'(r) - \phi(r)$ for $r \in \R$, the total free energy associated to $\phi$ is defined as
		\begin{multline}
			\H^{\text{tot}}(t) = \H(t) + R\pi_\phi\left(\frac{\alpha_1}{\beta_1}\right) (X_1(t)-X_1(0)) \\
			- \pi_\phi\left(\frac{\alpha_0}{\beta_0}\right) \int_0^t (\alpha_0-\beta_0 u(t,X_0(t)) - \phi'\left(\frac{a}{b}\right) \int_0^t (a - b u(t, X_0(t))).
		\end{multline}
	\end{definition}
	\noindent
	\begin{remark} \label{rk:pi_croiss}
		Since $\pi_\phi'(r) = r \phi''(r)$ with $\phi$ convex, $\pi_\phi$ is non-decreasing over $\R_{\geq 0}$.
	\end{remark}
	\noindent
	The following proposition is proved with formal computations, assuming that $u \geq 0$,  $X_0$ and $X_1$ exist and are regular enough. 
	\begin{proposition} \label{prop:energy_dec}
		 Assume that there exists a solution $(u, X_0, X_1)$ to the system \eqref{eq:sys_mov} such that $u$ is non-negative and $X_1-X_0$ positive and let $\phi \in \mathcal{C}^2(\R)$. Then, there exists a non-negative function $\mathcal{D}_\phi$ such that, for any $t > 0$,
		\begin{equation*}
			\dfrac{d}{dt} \H^{\text{tot}}(t) = -\mathcal{D}_\phi(t) \leq 0.
		\end{equation*}
	\end{proposition}
	\begin{proof} 
		Let $t > 0$. Compute the derivative of $\H(t)$ and use the equation \eqref{eq:u} to get
		\begin{equation*}
			\frac{d}{dt}\H(t) = - \intmov \phi'(u)\d{x}J + X'_1(t) \phi(u(t,X_1(t))) - X'_0(t) \phi(u(t,X_0(t))).
		\end{equation*}
		Integrating by parts and using the relation between the fluxes $J$ and $\Flux$ at the boundaries \eqref{eq:rel_flux} yields
		\begin{multline*}
			\frac{d}{dt}\H(t) = -\intmov \lvert \d{x}u \rvert^2\phi''(u) - RX'_1(t)\pi_\phi(u(t,X_1(t))) \\
	+ (X'_0(t) - (1-R)X'_1(t))\pi_\phi(u(t,X_0(t))) + \Flux_0 \phi'(u(t,X_0(t))).
		\end{multline*}
		Then, using the equations on the moving boundaries \eqref{eq:X0_cont}-\eqref{eq:X1_cont}, one gets that
		\begin{equation*}
			\frac{d}{dt}\H^{\text{tot}}(t) = - \mathcal{D}_\phi(t),
		\end{equation*}
		where the dissipation term $\mathcal{D}_\phi(t)$ is 
		\begin{multline*} \label{eq:diss_term}
			\mathcal{D}_\phi(t) = \intmov \vert \d{x}u \vert ^2 \phi''(u) - R(\alpha_1 - \beta_1 u(t,X_1(t))) \left(\pi_\phi(u(t,X_1(t))) - \pi_\phi \left( \frac{\alpha_1}{\beta_1} \right) \right) \\
		- (\alpha_0 - \beta_0 u(t,X_0(t))) \left( \pi_\phi(u(t,X_0(t))) - \pi_\phi \left( \frac{\alpha_0}{\beta_0} \right) \right) \\
		- (a - b u(t, X_0(t))) \left( \phi'(u(t,X_0(t))) - \phi' \left( \frac{a}{b} \right) \right).
		\end{multline*} 
		
		Since $\pi_\phi$ and $\phi'$ are non-decreasing over $\R_{\geq 0}$ (see Remark \ref{rk:pi_croiss}), it easily comes that all the terms in $\mathcal{D}_\phi(t)$ are non-negative, and so is $\mathcal{D}_\phi$.
	\end{proof}
	
	This result gives various a priori bounds on any solution to the problem. The following proposition presents the a priori $L^\infty$ bound obtained thanks to the energy estimates.
	\begin{proposition} \label{prop:u_bounds}
		Assume that there exists a solution $(u, X_0, X_1)$ to the system \eqref{eq:sys_mov} such that $u$ is non-negative and $X_1-X_0$ positive. Then, under the assumption \eqref{hyp}, $u$ satisfies
		\(
			m \leq u \leq M, 
		\)
		where $m$ and $M$ are defined by
		\begin{equation} \label{eq:def_Linf_bounds}
			m = \min \left \{\operatorname{ess-inf} u^\text{init}, \frac{\alpha_1}{\beta_1}\right \}, \: M = \max \left \{\operatorname{ess-sup} u^\text{init}, \frac{\alpha_0}{\beta_0} \right \}.
		\end{equation}
	\end{proposition}
	\begin{proof}
		Take $\phi \in \mathcal{C}^2(\R)$ convex such that $\phi = 0$ on $[m,M]$, and $\phi > 0$ everywhere else. Then, $\phi'$ and $\pi_\phi$ are identically zero in $[m,M]$, thus Assumption \eqref{hyp} ensures that $\phi(r) = \phi'(r)= \pi_\phi(r)=0$ for $r \in \left \{ \frac{\alpha_1}{\beta_1}, \frac{a}{b}, \frac{\alpha_0}{\beta_0}\right\}$. Hence for all $t \in [0,T]$, there holds $\H^{\text{tot}}(t) =\H(t)$ and
		Proposition \ref{prop:energy_dec} gives that $\frac{d}{dt}\H \leq 0$. Since $u^\text{init} \in [m, M]$, it comes that $\H(0) = 0$. Hence $\H(t) \leq 0$ for $t \geq 0$. Since $\phi \geq 0$ the definition~\eqref{eq:H} of $\H$ ensures that $\H(t) \geq 0$, thus $\H(t) = 0.$
		Finally, $\phi$ being non-negative and continuous, it comes that \(\phi(u) \equiv 0\), ensuring that  $m \leq u \leq M$.
	\end{proof}
\section{Analysis on a fixed mesh}
As a first contribution to the analysis of the numerical scheme, we establish some properties of the scheme on a given mesh, i.e. without letting the discretization parameters tend to $0$. 
	\subsection{Preservation of the travelling wave profile} \label{sec:tw}
	In this part, it is proved that the travelling wave profile is a (quasi-)stationary state for the scheme.
	\begin{proposition}
		The numerical scheme \eqref{eq:scheme} preserves the travelling wave profile in the following sense: for $0 \leq i \leq I+1$ and $0 \leq n \leq N$, denote $u_i^n = \hat{u}(\xi_i)$ with $\hat{u}$ defined in \eqref{eq:tw_u}, set $X_1^n = \hat{L} + \hat{c}n\dt$ and $X_0^n = \hat{c} n \dt$.
Then, the sequence $((u_i^n)_{0 \leq i \leq I+1}, X_0^n, X_1^n)_{0 \leq n \leq N}$ satisfies the scheme \eqref{eq:scheme}.
	\end{proposition}
	\begin{proof}
		Let $0 \leq i \leq I+1$ and $1 \leq n \leq N$. Denote $u_i^n = \hat{u}(\xi_i)$, and set $X_1^n = \hat{L} + \hat{c}n\dt$ and $X_0^n = \hat{c} n \dt$.
		The definitions \eqref{eq:tw_size}, \eqref{eq:tw_velocity} of $\hat{c}$ and $\hat{L}$ ensure that the equations on the boundaries in the scheme \eqref{eq:X1} and \eqref{eq:X0} are satisfied.
		
		\noindent
		Remark that $\v{i+}{n} = -R \hat{c}$. Then, the property \textit{(ii)} of the Bernoulli function given in Lemma \ref{lem:Bern}, combined with the exponential profile ${\hat u}$, implies the cancellation of all the numerical fluxes.  This yields  \eqref{eq:scheme_u}, \eqref{eq:left_cond_flux}, \eqref{eq:right_cond_flux}  and concludes the proof.
	\end{proof}
	\subsection{Discrete energy estimates}
	To begin with, one can prove a result of decreasing energy that can be seen as the discrete counterpart of Proposition \ref{prop:energy_dec}. Thus, it is necessary to define a discrete free energy as follows.
	\begin{definition}
		Let $\phi \in \mathcal{C}^2(\R)$ an arbitrary convex function. Set $n \geq 0$. The discrete free energy associated to $\phi$ is given by
		\begin{equation*}
			\H^n = \sum_{i=1}^I L^n h_i \phi(u_i^n).
		\end{equation*}
		Moreover, the discrete total free energy associated to $\phi$  is
		\begin{multline} \label{eq:def_htot_disc}
			\H^{{\text{tot}},n} = \H^n + \pi_\phi \left( \frac{\alpha_0}{\beta_0} \right) \sum_{k=1}^n \dt (\alpha_0 - \beta_0 u_0^k) \\
			+ \phi' \left(\frac{a}{b} \right) \sum_{k=1}^n \dt (a - b u_0^k) + R\pi_\phi \left( \frac{\alpha_1}{\beta_1} \right) \sum_{k=1}^n \dt (\alpha_1 - \beta_1 u_{I+1}^k).
		\end{multline}
	\end{definition}

	\begin{proposition} \label{prop:disc_energ_dec}
	    Assume that the scheme \eqref{eq:scheme} admits a solution $((u_i^n)_{0 \leq i \leq I+1},$ $X_0^n, X_1^n)_{0 \leq n \leq N} $ such that, for all $0\leq n\leq N$, $L^n>0$ and $u_i^n\geq 0 $ for all $0\leq i\leq I+1$.
		Then, for any convex function $\phi \in \mathcal{C}^2(\R)$, there exists a non-negative dissipation term $\mathcal{D}_{\phi}^n$ such that the discrete derivative of $\H^{{\text{tot}},n}$ satisfies
		\begin{equation} \label{eq:disc_htot_decr}
			\dfrac{\H^{{\text{tot}},n} - \H^{{\text{tot}},n-1}}{\dt} + \mathcal{D}_{\phi}^n \leq 0.
		\end{equation}
	\end{proposition}
	\begin{proof}
	Let us introduce the application $A: (l,w) \mapsto l \phi\left(\dfrac{w}{l}\right)$ in order to rewrite 
	$$
	\H^n=\sum_{i=1}^I h_i A\left(L^n, L^n u_i^n \right).
	$$
	As $\phi$ is convex, $A$ is jointly convex and verifies $\nabla A(l,w)=(-\pi_\phi(\dfrac{w}{l}), \phi'(\dfrac{w}{l}))$. 
	From an elementary convexity inequality, we deduce that 
		\begin{equation}\label{disc_energ_dissip_startpt}
			\frac{\H^n - \H^{n-1}}{\dt} \leq Y^n + Z^n
		\end{equation}
		with
		\[
			Y^n = \sum_{i=1}^I h_i \phi'(u_i^n) \frac{L^n u_i^n - L^{n-1} u_i^{n-1}}{\dt}, \qquad 
			Z^n =  -\delta^n[L] \sum_{i=1}^I h_i  \pi_\phi(u_i^n).
		\]
		Using the numerical scheme \eqref{eq:scheme_u} in the expression $Y^n$, performing a discrete integration by parts and using the 
		discrete boundary conditions~\eqref{eq:left_cond_flux} and \eqref{eq:right_cond_flux} provides 
		\begin{equation}\label{def:Yn}
			Y^n  = (a-b u_0^n) \phi'(u_0^n) + \sum_{i=0}^{I} \left( \phi'(u_{i+1}^n) - \phi'(u_{i}^n) \right) \Fluxnum{i+}.
		\end{equation}
		For any $0 \leq i \leq I$ and $n \geq 0$, the mean value formula and the link between $\phi$ and $\pi_\phi$ 
		ensure the existence of $\theta_{i+\frac{1}{2}}^n \in [0,1]$ such that
		\begin{equation} \label{eq:mean_value}
			\pi_\phi(u_{i+1}^n) - \pi_\phi(u_{i}^n) = \left(\theta_{i+\frac{1}{2}}^n u_{i}^n + (1-\theta_{i+\frac{1}{2}}^n)u_{i+1}^n\right) \left(\phi'(u_{i+1}^n) - \phi'(u_{i}^n)\right).
		\end{equation}
	Then, using the third property of the Bernoulli function given in \eqref{eq:Bern} in Lemma \ref{lem:Bern}, it comes that
		\begin{equation} \label{eq:Yn-2}
			\sum_{i=0}^I\left( \phi'(u_{i+1}^n) - \phi'(u_{i}^n) \right) \Fluxnum{i+} =
			- \mathcal{D}_{\phi}^{\text{bulk},n} + \sum_{i=0}^I v_{i + \frac{1}{2}}^n (\pi_\phi(u_{i+1}^n) - \pi_\phi(u_{i}^n)),
		\end{equation}
		where  the quantity $\mathcal{D}_{\phi}^{\text{bulk},n}$ is defined by
		\begin{multline}\label{eq:dbulk.sum}
		\mathcal{D}_{\phi}^{\text{bulk},n} = \sum_{i=0}^{I}
		\left(\Bern{i+}{n}{} \theta_{i+\frac{1}{2}}^n + \Bern{i+}{n}{-}(1 - \theta_{i+\frac{1}{2}}^n) \right)\times  \\
			\frac{\phi'(u_{i+1}^n) - \phi'(u_{i}^n) }{L^n \h{i+}} (u_{i+1}^n - u_{i}^n)  \geq 0.
	\end{multline}
	Thanks to a discrete integration by parts, the last term in \eqref{eq:Yn-2} rewrites 
	$$
	\sum_{i=0}^I v_{i + \frac{1}{2}}^n (\pi_\phi(u_{i+1}^n) - \pi_\phi(u_{i}^n))=-\sum_{i=1}^I (v_{i + \frac{1}{2}}^n-v_{i - \frac{1}{2}}^n)\pi_\phi (u_{i}^n) + v_{I+\frac12}^n \pi_\phi(u_{I+1}^n) -v_{\frac{1}{2}}^n\pi_\phi(u_0^n).
	$$
	The definition \eqref{eq:def_v} of the discrete velocities implies that \((v_{i + \frac{1}{2}}^n-v_{i - \frac{1}{2}}^n)=-\delta^n[L]h_i\) and also fix values of  $v_{I+\frac{1}{2}}^n$ and $v_{\frac{1}{2}}^n$,
	so that 
	\begin{multline}\label{est_termesadd}
	\sum_{i=0}^I v_{i + \frac{1}{2}}^n (\pi_\phi(u_{i+1}^n) - \pi_\phi(u_{i}^n))=-Z^n-R\delta^n[X_1]\pi_\phi (u_{I+1}^n)\\
	-((1-R)\delta^n[X_1]-\delta^n[X_0])\pi_\phi(u_0^n).
	\end{multline}
		Then, using the equations on $\delta^n[X_0]$ and $\delta^n[X_1]$ in \eqref{eq:scheme}, we deduce from \eqref{disc_energ_dissip_startpt}, \eqref{def:Yn}, \eqref{eq:Yn-2} and \eqref{est_termesadd}:
		\begin{multline*}
			\dfrac{\H^n - \H^{n-1}}{\dt} + \mathcal{D}_{\phi}^{\text{bulk},n} \\
			+ (bu_0^n - a)\phi'(u_0^n)+ (\beta_0 u_0^n-\alpha_0 )\pi_\phi(u_0^n) 
			 + R( \beta_1 u_{I+1}^n-\alpha_1)\pi_\phi(u_{I+1}^n) \leq 0.
		\end{multline*}
		
		In order to recover the expression of $\H^{{\text{tot}},n}$ given in \eqref{eq:def_htot_disc}, it only remains to use the non-decreasing behavior of $\phi'$ and $\pi_\phi$ (see Remark \ref{rk:pi_croiss}). Then, the following term is non-negative:
		\begin{multline}\label{eq:dbound}
			\mathcal{D}^{\text{bound}, n}_\phi = (\beta_0 u_0^n-\alpha_0 )\left(\pi_\phi(u_0^n) - \pi_\phi\left(\frac{\alpha_0}{\beta_0}\right)\right) \\
			+ (bu_0^n-a)\left(\phi'(u_0^n) - \phi'\left(\frac{a}{b}\right)\right) \\
			+R(\beta_1 u_{I+1}^n-\alpha_1)\left(\pi_\phi(u_{I+1}^n) - \pi_\phi\left(\frac{\alpha_1}{\beta_1}\right)\right) \geq 0.
		\end{multline}
		Finally, the  inequality \eqref{eq:disc_htot_decr} holds
		with the dissipation term $\mathcal{D}_{\phi}^n$ defined as 
		\begin{equation} \label{eq:disc_diss_term}
			\mathcal{D}_{\phi}^n =  \mathcal{D}_{\phi}^{\text{bulk},n} + \mathcal{D}^{\text{bound}, n}_\phi.
		\end{equation}
	\end{proof}
	This result will be useful in the rest of the paper to get many a priori estimates on any solution to the numerical scheme, thanks to different choices of energy density function $\phi$.

	\subsubsection{Uniform bounds on the concentrations}	
	
	Let us establish the same a priori $L^{\infty}$ bounds as for the continuous problem.
	\begin{proposition} \label{th:bornes_apr}
	    Assume that the scheme \eqref{eq:scheme} admits a solution $((u_i^n)_{0 \leq i \leq I+1},$  $X_0^n, X_1^n)_{0 \leq n \leq N}$ such that, for all $0\leq n\leq N$, $L^n>0$ and $u_i^n\geq 0 $ for all $0\leq i\leq I+1$. Then, the solution to the scheme satisfies 
		\begin{equation} \label{eq:th_bornes_apr}
			m \leq u_i^n \leq M \quad \forall 0\leq i\leq I+1,\ \forall 0\leq n\leq N
		\end{equation}
		where $m$ and $M$ are defined in \eqref{eq:def_Linf_bounds}.
	\end{proposition}
	To prove this result, we adapt to the discrete setting the proof of Proposition~\ref{prop:u_bounds}.
	\begin{proof}
		The result is shown by induction on the time step $n \geq 0$. The property \eqref{eq:th_bornes_apr} being clearly verified for $n=0$,
     set $n > 0$ and assume that for any $0 \leq i \leq I+1$, $u_i^{n-1}$ verifies \eqref{eq:th_bornes_apr}. 
     
     As in the proof of Proposition \ref{prop:u_bounds}, we take $\phi \in \mathcal{C}^2(\R)$ convex such that $\phi = 0$ on $[m,M]$, $\phi > 0$ and strictly convex everywhere else and we apply Proposition \ref{prop:disc_energ_dec}.
	 For such a function $\phi$, one has 
	 $\H^{{\text{tot}},n} = \H^n$ and $\H^{{\text{tot}},n-1} = \H^{n-1}$. Moreover, $\H^{n-1} = 0$ thanks to the induction hypothesis. Then, one gets that
    \begin{equation*}
		\H^n + \dt \mathcal{D}^n_{\phi} \leq 0.
    \end{equation*}
Since both $\H^n$ and $\mathcal{D}^n_{\phi}$ are also both non-negative, we deduce that \(\H^n = 0\), which yields $\phi(u_i^n) = 0$ and therefore $m\leq u_i^n\leq M$ for $1\leq i\leq I$, and \(\mathcal{D}^n_{\phi} = 0,\) which yields the same bounds for $u_0^n$ and $u_{I+1}^n$.
	\end{proof}
	
One then readily deduces the following statement from~\eqref{eq:X0}--\eqref{eq:L} and \eqref{eq:def_v}.
	\begin{corollary}\label{coro:bounds.vitesse} Under the assumptions of Proposition~\ref{th:bornes_apr}, there holds
	\begin{align}\label{eq:bound.deltaX1}
	 - \alpha_1 + \beta_1 m  \leq &\; \delta^n[X_1] \leq - \alpha_1 + \beta_1 M, \\
	 - \alpha_0 - R\alpha_1 + m(\beta_0 + R \beta_1)  \leq &\; \delta^n[L] \leq   - \alpha_0 - R\alpha_1 + M(\beta_0 + R \beta_1). 
	 \label{eq:bound.deltaL}
	\end{align}
	As a consequence, there exists some $v_\flat$ and $v^\sharp$ depending only on the data of the continuous problem such that 
	\begin{equation}\label{eq:bound.v}
	v_\flat \leq v_{i+\frac12}^n \leq v^\sharp, \qquad \forall \, i \in \{0,\dots, I\}.
	\end{equation}
	\end{corollary}

	\subsubsection{Bounds on the domain size}
	In this section, two different kinds of estimates are established on the domain size. The first one presented below in Proposition \ref{prop:Ln_bounds_non_unif}, is a non-uniform estimate with respect to the discretization parameters, but it is global in time. It will be useful to prove the existence of a solution to the scheme \eqref{eq:scheme}.
	\begin{proposition} \label{prop:Ln_bounds_non_unif}
		Suppose that the hypothesis \eqref{hyp} is satisfied, and that there exists a solution $((u_i^n)_{0 \leq i \leq I+1}, X_0^n, X_1^n)$ to the numerical scheme \eqref{eq:scheme} with $L^n > 0$  and $u_i^n\geq 0 $ for all $0\leq i\leq I+1$. Then $L^n > \min \left( \frac{m}{M} L^{n-1}, \hat{L} \right)$.
	\end{proposition}
	\begin{proof}
		Let $n \geq 1$. Suppose for contradiction that $L^n \leq \min \left( \frac{m}{M} L^{n-1}, \hat{L} \right)$. This implies $\delta^n[L]< 0$ since Assumption~\eqref{hyp} 
		ensures that $m < M$. Furthermore,  owing to the $L^{\infty}$-bounds \eqref{eq:th_bornes_apr}, one has
		\begin{equation*}
			h_i \frac{L^n u_i^n - L^{n-1} u_i^{n-1}}{\dt} \leq 0, \quad \forall 1 \leq i \leq I.
		\end{equation*}
		Using the scheme \eqref{eq:scheme}, it comes that 
		\begin{equation} \label{eq:flux_neg}
			\Fluxnum{i-}\leq \Fluxnum{i+} \leq \Fluxnum{I+} = 0,\quad \forall 1\leq i\leq I.
		\end{equation}
		In particular, $\Fluxnum{} = a - b u_0^n \leq 0$, that is $u_0^n \geq \dfrac{a}{b}$. This implies
		\begin{equation}\label{eq:vi+1/2.estime}
			\v{i+}{n} = - \alpha_0 + \beta_0 u_0^n - \delta^n[L] \xi_{i+\frac{1}{2}} \geq  -\alpha_0 + \beta_0 \frac{a}b = - R \hat c.
		\end{equation}
		Moreover, one infers from \eqref{eq:flux_neg}, from the definition~\eqref{eq:def_flux_num} of the numerical fluxes, and from property (ii) of Lemma~\ref{lem:Bern} that 
		\(
		u_{i+1}^n \geq u_i^n \exp( L^n h_{i+\frac12} \v{i+}{n}), 
		\)
		which yields
		\begin{equation*}
		u_{I+1}^n\geq \exp{\left(\sum_{i=0}^{I}L^n \h{i+} \v{i+}{n}\right)} u_0^n \geq u_0^n e^{-L^n R \hat c} \geq \frac{a}{b} e^{- \hat{L} R \hat c} = \hat{u}(1). 
		\end{equation*}

		The relation \eqref{eq:X1} then provides that \(\delta^n[X_1] \geq \hat c\). But, noticing that \eqref{eq:X0} combined with \eqref{eq:L} rewrites 
		\begin{equation*}\label{eq:deltaL}
		 \delta^n[L]=R\delta^n[X_1]-(\alpha_0-\beta_0 u_0^n),   
		\end{equation*} 
		we then deduce that \(\alpha_0-\beta_0 u_0^n > R \hat c \), which contradicts $u_0^n \geq \dfrac{a}b$ obtained thanks to \eqref{eq:flux_neg}.
		Therefore, the starting hypothesis in this proof can not be satisfied, which concludes the proof of Proposition~\ref{prop:Ln_bounds_non_unif}. 
	\end{proof}
	Now, the second bound presented in this section is uniform with respect to the discretization, but only up to some finite time horizon $T$. As it is a direct consequence of \eqref{eq:bound.deltaL}, we do not detail the proof here.
	\begin{proposition} \label{prop:Ln_bd}
		Suppose that the hypothesis \eqref{hyp} is satisfied, and that there exists a solution $((u_i^n)_{0 \leq i \leq I+1}, X_0^n, X_1^n)$ to the numerical scheme \eqref{eq:scheme}  with  $u_i^n\geq 0 $ for all $0\leq i\leq I+1$ for $T < \dfrac{L^0}{(\alpha_0 + R \alpha_1) - m(\beta_0 + R \beta_1) }$. Then, there exist two positive constants $\underline{L}$ and $\overline{L}$, depending only on $T$, $L^0$, $m$, $M$ and the kinetic parameters, such that for all $n \in \{1,\dots, N\}$, it holds
		\begin{equation} \label{eq:unif_bds_Ln}
			0 < \underline{L} \leq L^n \leq \overline{L}.
		\end{equation}
	\end{proposition}

	\subsection{Existence of a solution to the numerical scheme} \label{sec:exist_scheme}
	The main result of this section can finally be proved.
		\begin{theorem} \label{th:exist_scheme}
			Under the assumption \eqref{hyp}, for a given mesh and a given time step, the numerical scheme \eqref{eq:scheme} admits a solution $((u_i^n)_{0 \leq i \leq I+1}, X_0^n, X_1^n)$ for any $0 \leq n \leq N$.
		\end{theorem}
				\begin{proof}
The proof relies on a Brouwer topological degree argument (see for instance Chapter 1, Section 3.1 in \cite{Deimling}). Our goal is then to construct a homotopy to pass from an invertible linear system for $\lambda = 0$ to the system~\eqref{eq:scheme} for $\lambda = 1$ in such a way that the solution $((u_i^{(\lambda)})_{0 \leq i \leq I+1}, X_1^{(\lambda)}, L^{(\lambda)})$ for $\lambda \in [0,1]$ remains in some open set $\Omega$ containing the solution for $\lambda = 0$ and such that the scheme is uniformly continuous on $\overline \Omega$. Note that we use $((u_i^{(\lambda)})_{0 \leq i \leq I+1}, X_1^{(\lambda)}, L^{(\lambda)})$ as main unknowns there, and that $X_0^{(\lambda)}$ is deduced a posteriori. 

For $\lambda \in [0,1]$, we seek $((u_i^{(\lambda)})_{0 \leq i \leq I+1}, X_1^{(\lambda)}, L^{(\lambda)})$ solution to 
\begin{align*}
\frac{\lambda (L^{(\lambda)} - L^{n-1}) u_i^{(\lambda)} + L^{n-1}(u_i^{(\lambda)} - u_i^{n-1})}{\dt} h_i + \lambda\left( \Ff_{i+\frac12}^{(\lambda)} -  \Ff_{i-\frac12}^{(\lambda)} \right) &\;= 0, \quad 1 \leq i \leq I,  \\
\lambda \left( \Ff_{\frac12}^{(\lambda)} - a + b u_0^{(\lambda)} \right) + (1-\lambda) \left( \beta_0  u_0^{(\lambda)} - \alpha_0 \right) &\; = 0, \\
\lambda \Ff_{I+\frac12}^{(\lambda)} + (1-\lambda) \left( \beta_1 u_{I+1}^{(\lambda)} - \alpha_1 \right) & \; = 0, \\
R \frac{X_1^{(\lambda)} - X_1^{n-1}}{\dt} - \frac{L^{(\lambda)} - L^{n-1}}{\dt} - \alpha_0 + \beta_0 u_0^{(\lambda)} &\;= 0, \\
\frac{X_1^{(\lambda)} - X_1^{n-1}}{\dt} + \alpha_1 - \beta_1 u_{I+1}^{(\lambda)} &\;= 0, 
\end{align*}
with 
\[
 \Ff_{i+\frac12}^{(\lambda)}  = \frac{1}{L^{(\lambda)}\h{i+}} \left( B\left(-L^{(\lambda)} h_{i+\frac12} v_{i+\frac12}^{(\lambda)} \right)u_i^{(\lambda)} -B\left(L^{(\lambda)} h_{i+\frac12} v_{i+\frac12}^{(\lambda)} \right)u_{i+1}^{(\lambda)} \right)
\]
and 
\[
 v_{i+\frac12}^{(\lambda)} = - R  \frac{X_1^{(\lambda)} - X_1^{n-1}}{\dt} \xi_{i+\frac12} + (1 - \xi_{i+\frac12}) ( \beta_0 u_0^{(\lambda)} - \alpha_0).
\]

For $\lambda = 0$, the unique solution to the above system is $u_i^{{(0)}} = u_i^{n-1}$ for $1 \leq i \leq I$, $u_0^{(0)} = \frac{\alpha_0}{\beta_0}$, $u_{I+1}^{(0)} = \frac{\alpha_1}{\beta_1}$, $X_1^{(0)} = X_1^{n-1}$ and $L^{(0)}= L^{n-1}$. Moreover, one can extend the proof of Proposition~\ref{prop:disc_energ_dec} and then Proposition~\ref{th:bornes_apr}, Corollary~\ref{coro:bounds.vitesse}, and Proposition~\ref{prop:Ln_bounds_non_unif} to the case $\lambda \leq 1$.
In particular, one checks that 
\begin{align*}
m \leq\; & u_i^{(\lambda)} \leq M, \\ 0 < \min\left( \hat L, \frac{m}M L^{n-1}\right) \leq\; & L^{(\lambda)} \leq L^{n-1} + \dt ( M(\beta_0 + R \beta_1) - \alpha_0 - R\alpha_1), \\
 X_1^{n-1} + \dt \left( \beta_1 m - \alpha_1 \right) \leq\; & X_1^{(\lambda)} \leq X_1^{n-1} + \dt \left( \beta_1 M - \alpha_1 \right).
\end{align*}
These uniform bounds allow to use Brouwer's topological degree to ensure that the system admits (at least) a solution for all $\lambda \in [0,1]$. 
Note here that the positive lower bound on $L^{(\lambda)}$ is crucial to ensure the uniform continuity of the flux $\Ff_{i+\frac12}^{(\lambda)}$. 
		\end{proof}

\section{Convergence of the scheme} \label{sec:conv}

The goal of this part is to establish rigorously, thanks to compactness and consistency arguments, that our scheme captures 
the solution to the continuous problem~\eqref{eq:sys_mov}.

		\subsection{Weak and approximate solutions} \label{sec:weak}

As our goal is to establish the convergence of the discrete solution, the existence of which being guaranteed thanks to Theorem~\ref{th:exist_scheme}, one needs to first give a precise sense to the notion of solution to the continuous model~\eqref{eq:sys_mov}. This is the purpose of the following definition. 

		\begin{definition} \label{def:weak}
			A triplet $(u,X_0,X_1)$ is said to be a \emph{weak solution} of \eqref{eq:sys_mov} up to time $T > 0$ if it satisfies:
	\begin{enumerate}[label=(\textit{\roman*})]
		\item $X_0$ and $X_1$ belong to $\mathcal{C}^{0,1} \left( [0,T), \R \right)$, with $X_0(0) = 0$, $X_1(0)= L^0$ and $L(t)=X_1(t)-X_0(t)>0$ for all $t\in [0,T)$ ;
		\item $u \in L^{\infty}(Q_T)$, with $u \geq 0$ a.e. in $Q_T$, and $\d{x}u \in L^2(Q_T)$;
		\item For any $\psi \in \mathcal{C}^\infty_c (Q_T)$, there holds
			\begin{multline}\label{eq:weak}
				\int_{X_0(0)}^{X_1(0)} u^\text{init}\psi(0,\cdot) + \int_0^T \int_{X_0(t)}^{X_1(t)} (u \d{t} \psi + (- \d{x}u + (1-R)uX_1') \d{x} \psi)\\ + \int_0^T \left(a - bu(\cdot, X_0(\cdot))\right) \psi(\cdot,X_0(\cdot)) = 0;
			\end{multline}
		\item For a.e. $t \in [0,T)$, $X_0$ and $X_1$ satisfy
		\[
		X_1'(t) = - \alpha_1 + \beta_1 u_1(t) \qquad \text{and} \quad 
		X_0'(t) = \alpha_0 - \beta_0 u_0(t) + (1-R) X_1'(t),
		\]
			where, for $\ell \in \{0,1\}$, $u_\ell(t)$ denotes the trace of $u(t, \cdot)$ at $X_\ell(t)$.
	\end{enumerate}				
		\end{definition}

\begin{remark}\label{rk:def-weak-sol}
Applying the change of variables 
$$
(t,\xi)\in [0,T)\times [0,1]\mapsto (t,x=X_0(t)+\xi L(t))\in Q_T,
$$
we obtain an equivalent definition of the weak solution. Indeed, $(u,X_0,X_1)$ is a weak solution of \eqref{eq:sys_mov} in the sense of Definition~\ref{def:weak} if and only if  $X_0$ and $X_1$ satisfy {\em (i)} and there exists  a positive function $w\in L^{\infty} ([0,T)\times [0,1])\cap L^2([0,T]\times H^1(0,1))$ such that
 \begin{itemize}
 \item For any $\varphi\in \mathcal{C}^\infty_c ([0,T)\times [0,1])$, there holds
 \end{itemize}
  \begin{multline}\label{eq:weak-w}
				\int_0^T \int_{0}^{1} \left(Lw \d{t} \varphi + \left(- \frac{1}{L}\d{\xi}w+ ((1-R)X_1'-L' \xi-X_0')w\right) \d{\xi} \varphi\right)(t,\xi) \dd t\dd \xi\\
				+\int_{0}^{1} L(0)u^\text{init}(0,L(0)\xi)\varphi(0,\xi)\dd \xi + \int_0^T \left(a - bw(t, 0)\right) \varphi(\cdot,0)\dd t = 0;
			\end{multline}
\begin{itemize}
\item  $X_1'(t)=-\alpha_1+\beta_1 w(t,1)$ and $X_0'(t)=\alpha_0-\beta_0 w(t,0) +(1-R) X_1'(t)$;
\item and $u(t,x)= w(t,\dfrac{x-X_0(t)}{L(t})$.
\end{itemize}

\end{remark}

In order to address the convergence of the scheme, let us now explain how we define some approximates solutions based on the discrete data 
$\left((u_i^n)_{0 \leq i \leq I+1}, X_0^n, X_1^n\right)_{0 \leq n \leq N}$. 
For the boundaries, we set 
\begin{equation}\label{eq:Xhdt}
X_{\ell, h,\dt}(t) = X_\ell^n \quad \text{if}\; t \in (t^{n-1}, t^n], \quad \text{and}\quad X_{\ell, h,\dt}(0) = X_\ell^0, \qquad \ell \in \{0,1\},
\end{equation}
so that we can also define $L_{h,\dt} = X_{1,h,\dt}- X_{0,h,\dt}$. Then, it is possible to reconstruct the approximate densities on  $\bigcup_{t \in [0,T)} \Big(\{t\} \times [X_{0,h,\dt}(t), X_{1,h,\dt}(t)] \Big)$ by setting 
$$
u_{h,\Delta t} (t,x) = 
\left\{\begin{aligned}
& u_i^n\mbox{ if } (t,x)\in (t^{n-1},t^n]\times (x_{i-\frac{1}{2}}^n, x_{i+\frac{1}{2}}^n) \\
& u_i^0 \mbox{ if } t=0 \mbox{ and }x\in(x_{i-\frac{1}{2}}^0, x_{i+\frac{1}{2}}^0) \\
& u_\ell^n \mbox{ if } x=X_\ell^n \mbox{ and } t\in (t^{n-1},t^n], \mbox { for } \ell\in \{0, 1\}.
\end{aligned}\right. 
$$
This reconstruction is useful to show simultaneously the evolution of the oxide layer and the evolution of the density profile inside the layer, as on Figure~\ref{fig:profiles}. However, in order to compare different approximate solutions --when studying the long time behavior of the scheme, or when studying its convergence--, it seems relevant to reconstruct the approximate densities on the underlying fixed domain $[0,T]\times [0,1]$. This approximate solution will be denoted by $w_{h,\Delta t}$ and defined by 
 \begin{equation}\label{eq:whdt}
 w_{h,\Delta t}(t,\xi) = 
 \left\{\begin{aligned}
& u_i^n\mbox{ if } (t,\xi)\in (t^{n-1},t^n]\times (\xi_{i-\frac{1}{2}}, \xi_{i+\frac{1}{2}}) \\
& u_i^0 \mbox{ if } t=0 \mbox{ and }\xi \in(\xi_{i-\frac{1}{2}}, \xi_{i+\frac{1}{2}}) \\
& u_\ell^n \mbox{ if } \xi=\ell \mbox{ and } t\in (t^{n-1},t^n], \mbox { for } \ell\in \{0, 1\}.
\end{aligned}\right. 
\end{equation}
We will establish the convergence of the scheme by studying the convergence of $(w_{h,\Delta t},X_{0,h,\Delta t}, X_{1,h,\Delta t})$ when $h$ and $\Delta t$ tend to 0. As a by-product of the convergence of the scheme, we will obtain the existence of a weak solution to \eqref{eq:sys_mov} in the sense of Definition~\ref{def:weak} thanks to Remark~\ref{rk:def-weak-sol}.

		In the spirit of \cite{kangourou_2018}, we introduce the reconstruction operator $\Pi_{\mathcal{M}}$ corresponding to a 
		mesh $\mathcal M = {(\xi_{i+\frac12})}_{0 \leq i \leq I}$, mapping 
		$\R^{I+2}$ to bounded functions defined on $[0,1]$.
		 We avoid here the explicit introduction of the reconstruction operator in the notation, 
		but rather denote by $z_h = \Pi_{\mathcal{M}}\left(\boldsymbol z\right)$ the image of $\boldsymbol z =(z_i)_{0\leq i\leq I+1} \in \R^{I+2}$, defined by 
\[
\Pi_\Mm \bz (\xi) = z_i \textnormal{ for } \xi \in (\xi_{i-\frac{1}{2}}, \xi_{i+\frac{1}{2}}),
\quad 
\Pi_\Mm\bz (0) = z_0, \quad \text{and}\quad \Pi_\Mm\bz(1) = z_{I+1}. 
\]
We define the finite dimensional vector space
$S_{\mathcal M} = \operatorname{Range}(\Pi_{\mathcal M})$, 
equipped with discrete $H^1(0,1)$ and $(H^1)'(0,1)$ norms. More precisely, for all $z_h = \Pi_\Mm \bz \in S_\Mm$, we set
\begin{align*}
\normeh{z_h} =\;& \left(  \sum_{i=1}^{I+1} \frac{(z_i - z_{i-1})^2}{\h{i-}} + z_0^2 \right)^{\frac{1}{2}},\\
\normehs{z_h} =\;& \sup \left\{ \int_0^1 z_h \eta_h \textnormal{ / } \eta_h \in S_\mathcal{M}, \: \normeh{\eta_h} \leq 1 \right\}, 
\end{align*}
		We recall the discrete Poincaré inequality satisfied in this framework (\cite{CB08}):
		\begin{equation} \label{eq:poincare}
			\Vert z_h \Vert_{L^2(0,1)} \leq \sqrt{2} \normeh{z_h}, \qquad z_h \in S_\Mm. 
		\end{equation}

For time iterated discrete vectors $\bz = \left(\bz^n\right)_{0 \leq n \leq N} \in \left(\R^{I+2}\right)^{N+1}$, we define the linear reconstruction 
operator $\Pi_{\Mm,\dt}$ mapping $\left(\R^{I+2}\right)^{N+1}$ into the bounded functions on $[0,T] \times [0,1]$ by setting 
\[
 \Pi_{\Mm,\dt}\bz(t,\cdot) = \Pi_\Mm \bz^n \quad \text{if}\; t \in (t^{n-1}, t^n], \quad t \in (0,T], 
 \quad  \text{and} \quad \Pi_\Mm \bz^0(\xi) = z_h^0.
\]
The space $S_{\mathcal M,\dt} = \operatorname{Range}(\Pi_{\mathcal M,\dt})$ is equipped with the 
discrete $L^2(H^1)$ and $L^2(H^{-1})$ norms
\[
\normelh{ \Pi_{\Mm,\dt}\bz} = \left( \sum_{n=1}^N \dt \normeh{ \Pi_{\Mm}\bz^n}^2 \right)^{\frac{1}{2}},  \quad 
\normelhs{\Pi_{\Mm,\dt}\bz} = \left( \sum_{n=1}^N \dt \normehs{\Pi_{\Mm}\bz^n}^2 \right)^{\frac{1}{2}}.
\]

Denoting by $\bu = \left(\bu^n\right)_{0 \leq n \leq N} \in \left(\R^{I+2}\right)^{N+1}$ with $\bu^n = \left(u_i^n\right)_{0\leq i \leq I+1}$, we have that $w_{h,\dt}$ defined by \eqref{eq:whdt} verifies \(w_{h,\dt} = \Pi_{\Mm,\dt} (\bu)\) with \(w_h^n = \Pi_{\Mm}(\bu^n)\).

We also define the approximate time and space derivatives of $w_{h,\dt}$, respectively $\partial_{t,\dt} w_{h,\dt}$ and $\partial_{\xi,h} w_{h,\dt}$, by 
\begin{equation}\label{eq:dtwhdt}
\partial_{t,\dt} w_{h,\dt} = \begin{cases}
\frac{w_{h,\dt}(t) - w_{h,\dt}(t-\dt)}{\dt} & \text{for}\; t \in (\dt, T], \\
\frac{w_{h}^1 - w_h^0}{\dt} & \text{for}\; t \in [0,\dt]. 
\end{cases}
\end{equation}
and 
\begin{equation}\label{eq:dxiwhdt}
\partial_{\xi,h} w_{h,\dt}(t,\xi) = \frac{w_{i+1}^n - w_i^n}{h_{i+\frac12}} \quad \text{if}\; (t,\xi) \in (t^{n-1},t^n] \times (\xi_i, \xi_{i+1}).
\end{equation}
Note that $\partial_{t,\dt} w_{h,\dt}$ belongs to $S_{\Mm,\dt}$, but that $\partial_{\xi,h} w_{h,\dt}$ does not.

Finally, we define the approximate flux $\Ff_{h,\dt} \in L^\infty((0,T) \times (0,1))$ by 
\begin{equation}\label{eq:Fhdt}
\Ff_{h,\dt}(t,\xi) = \Ff_{i+\frac12}^n \quad \text{if}\; (t,\xi) \in (t^{n-1},t^n] \times (\xi_i, \xi_{i+1}).
\end{equation}

\subsection{Additional estimates} 
In this section, we derive some additional energy estimates on the approximate solution on a time interval $(0,T)$ with 
the time horizon $T$ being such that there exists some uniform bounds $L^\sharp \geq L_\flat >0$ for the domain size, i.e. $L^\sharp \geq L^n \geq L_\flat$ for all $n \in \{0,\dots, N\}$. 
The existence of such a positive $T$ is granted by Proposition~\ref{prop:Ln_bd}. 
	\begin{proposition} \label{prop:estim_u_l2}
	Assume that Assumption~\eqref{hyp} holds. 
	Let $T>0$ be a finite time horizon such that $L^\sharp > L^n \geq L_\flat>0$ for all $n  \in \{0,\dots, N\}$, then 
	there exists $\ctel{cte:L2H1}>0$ depending only on $L^\sharp, m, M$, on the parameters $a$, $b$, $\alpha_0$, $\beta_0$, $\alpha_1$, $\beta_1$, $R$ and $T$   (but neither on the mesh $\mathcal M$ nor on $\dt$) such that 
		$\normelh{w_{h,\dt}}^2 \leq \cter{cte:L2H1}.$
	\end{proposition}
	\begin{proof}
		Proposition \ref{prop:disc_energ_dec} gives
		\begin{equation}\label{eq:estim_u_12.1}
			\frac{\H^{\text{tot}, n} - \H^{\text{tot}, n-1}}{\dt} + \mathcal{D}_\phi^{\text{bulk},n} \leq - \mathcal{D}_\phi^{\text{bound},n} \leq 0, \qquad 1 \leq n \leq N,
		\end{equation}
		with \(\mathcal{D}_\phi^{\text{bulk},n}\) and \(\mathcal{D}_\phi^{\text{bound},n}\) being defined by~\eqref{eq:dbulk.sum} and \eqref{eq:dbound} 
		respectively. For the specific choice $\phi(r) = \frac12r^2$, we have 
		\[
		\mathcal{D}_\phi^{\text{bulk},n} \!=\! \sum_{i=0}^I \frac{(u_{i+1}^n - u_i^n)^2}{L^n \h{i+}}\!
		\left(\Bern{i+}{n}{-} \theta_{i+\frac{1}{2}}^n + \Bern{i+}{n}{}(1 - \theta_{i+\frac{1}{2}}^n) \right).
		\]
		Since $B$ is decreasing, one has $B(r) \geq B(|r|)$. Then since $\theta_{i+\frac12}^n \in [0,1]$ and since $L^n \leq L^\sharp$, 
		we obtain that 
		\[
		\mathcal{D}_\phi^{\text{bulk},n} \geq \frac1{L^\sharp}\sum_{i=0}^I \frac{(u_{i+1}^n - u_i^n)^2}{\h{i+}} B\left(L^n \h{i+} |v_{i+\frac12}^n|\right). 
		\]
		Thanks to the bounds on $v_{i+\frac12}^n$ stated in Corollary~\ref{coro:bounds.vitesse} 
		combined with the upper bound $L^\sharp$ on $L^n$ and with $\h{i+}\leq 1$, we get some uniform upper bound on $L^n \h{i+} |v_{i+\frac12}^n|$, 
		whence a uniform lower bound $\delta >0$ for $B\left(L^n \h{i+} |v_{i+\frac12}^n|\right)$. 
		As a consequence, estimate \eqref{eq:estim_u_12.1} yields
		\[
		 \frac\delta{L^\sharp} \|w_h^n\|^2_1 \leq \mathcal{D}_\phi^{\text{bulk},n} +  \frac\delta{L^\sharp} (u_0^n)^2\leq - \frac{\H^{\text{tot}, n} - \H^{\text{tot}, n-1}}{\dt}
		 + \frac\delta{L^\sharp} M^2.
		\]
		Then, multiplying by $\dt$ and summing over $n \in \{1,\dots, N\}$ gives
		\begin{equation} \label{eq:lhs_bd_norm}
		\|w_{h,\dt}\|_{2,1}^2  \leq  L^\sharp \frac{\mathcal{H}_\phi^{\text{tot},0} - \mathcal{H}_\phi^{\text{tot},N}}\delta +    M^2 T.
		\end{equation}
		We infer from the definition \eqref{eq:def_htot_disc} of the discrete total free energy that 
		\begin{equation}\label{eq:estim.Htot0}
		{\mathcal H}_\phi^{\text{tot},0} = {\mathcal H}_\phi^{0} \leq  {\mathcal H}_\phi(0) \leq \frac{L^\sharp M^2}2.
		\end{equation}
		On the other hand, since  $\mathcal{H}_\phi^{N} \geq 0$, one gets from~\eqref{eq:def_htot_disc} combined with the uniform upper 
		bound~\eqref{eq:th_bornes_apr} on $u_i^n$ that 
		\begin{equation}\label{eq:estim.HtotN}
		\mathcal{H}_\phi^{\text{tot},N} \geq - T \left( 
		\frac12 \frac{(\alpha_0)^2}{\beta_0}  \left(M-\frac{\alpha_0}{\beta_0} \right) 
			+ a \left(M-\frac{a}b \right) + \frac{R}2 \frac{(\alpha_1)^2}{\beta_1} \left(M-\frac{\alpha_1}{ \beta_1 }\right)\right).
		\end{equation}
		Taking~\eqref{eq:estim.Htot0} and \eqref{eq:estim.HtotN} into account in~\eqref{eq:lhs_bd_norm} concludes the proof.
	\end{proof}

Our next statement is about the control in the discrete $L^2((H^1)')$ norm of the discrete time derivative $\partial_{t,\dt} w_{h,\dt}$ introduced in~\eqref{eq:dtwhdt}.

	\begin{proposition} \label{prop:estim_du_hs}
		Under the assumptions of Proposition~\ref{prop:estim_u_l2}, there exist $\ctel{cte:FL2}>0$ and $\ctel{cte:L2H-1}>0$ depending only on $L_\flat$, $L^\sharp$, $m$, $M$, on the parameters $a$, $b$, $\alpha_0$, $\beta_0$, $\alpha_1$, $\beta_1$, $R$ and $T$, but not on the discretization, such that
		\begin{align}
	 		\| \Ff_{h,\dt} \|_{L^2((0,T) \times (0,1))} \leq& \; \cter{cte:FL2} \label{eq:FL2}\\
			\normelhs{\ddisc w_{h,\dt}} \leq &\;  \cter{cte:L2H-1} \label{eq:L2H-1}.
		\end{align}	
	\end{proposition}
	\begin{proof}
		Let $\eta_h \in S_\mathcal{M}$ with $\eta_h (\xi) = \eta_i$ for $\xi \in (\xi_{i-\frac{1}{2}}, \xi_{i + \frac{1}{2}})$ be such that $\normeh{\eta_h} \leq 1$, 
		then for $t \in (t^{n-1}, t^n]$, $1\leq n \leq N$, one has 
		\begin{equation*}
			E^n= \int_0^1 \ddisc w_{h,\dt}(t,\xi) \eta_h(\xi) \dd \xi = \sum_{i=1}^I h_i \frac{u^n_i - u^{n-1}_i}{\dt} \eta_i.
		\end{equation*}
		Rewriting 
		\[
		\frac{u^n_i - u^{n-1}_i}{\dt} = \frac1{L^{n}}\left( \frac{L^n u_i^n - L^{n-1} u_i^{n-1}}{\dt} - u_i^{n-1} \frac{L^n - L^{n-1}}{\dt} \right), 
		\]
		and using~\eqref{eq:scheme_u}, one gets that 
		\begin{equation}\label{eq:En=Fn+Gn}
		E^n= - \frac1{L^n}  \sum_{i=1}^I \eta_i \left(\Fluxnum{i+} - \Fluxnum{i-}\right) - \frac1{L^n}\frac{L^n - L^{n-1}}{\dt}   \sum_{i=1}^I \eta_i u_i^{n-1} h_i.
		\end{equation}
		As a direct consequence of~\eqref{eq:bound.deltaL}, 
		there exists some $C_4>0$ depending only on the data of the continuous problem such that  $\left|\delta^n[L]\right| \leq \ctel{cte:L'}$.
		The Poincaré inequality \eqref{eq:poincare}, the lower bound on $L^n$ and the upper bound on $u_i^{n-1}$ then yield
		\[
		\left|\frac1{L^n}\frac{L^n - L^{n-1}}{\dt}   \sum_{i=1}^I \eta_i u_i^{n-1} h_i \right| \leq \frac{\cter{cte:L'}}{L_\flat} {\|w_h^n\|}_{L^2(0,1)} {\| \eta_h \|}_{L^2(0,1)} \leq   \frac{\cter{cte:L'}}{L_\flat} \sqrt2 M. 
		\]
		Besides, applying a summation by parts to the first term in the right-hand-side of~\eqref{eq:En=Fn+Gn} provides 
		\[
		- \frac1{L^n}  \sum_{i=1}^I \eta_i \left(\Fluxnum{i+} - \Fluxnum{i-}\right) = \frac1{L^n} \sum_{i=0}^I \Fluxnum{i+} \left( \eta_{i+1} - \eta_i \right) + \frac1{L^n}  \Fluxnum{} \eta_0 .
		\]
		Since $| \Fluxnum{} | \leq \min(a - b m, bM-a)$, and since $|\eta_0| \leq \|\eta_h\|_1 \leq 1$, one obtains that 
		\[
		\left| \frac1{L^n}  \Fluxnum{} \,\eta_0 \right| \leq \frac{\min(a - b m, bM-a)}{L_\flat}.
		\]
		On the other hand, applying the Cauchy-Schwarz inequality, one gets that 
		\[
		 \frac1{L^n} \sum_{i=0}^I \Fluxnum{i+} \left( \eta_{i+1} - \eta_i \right) \leq \frac1{L^\flat} \left( \sum_{i=0}^I h_{i+\frac12} \left|\Fluxnum{i+} \right|^2 \right)^{\frac12}.
		\]
		Thanks to relation~\eqref{eq:Bern}, the Scharfetter-Gummel fluxes rewrite
		\begin{equation}\label{eq:flux.conv-diff}
		\Fluxnum{i+} = v_{i+\frac12}^n \frac{u_i^n + u_{i+1}^n}2 + \frac1{2L^n} \left( B(- L^n h_{i+\frac12} v_{i+\frac12}^n) + B(L^n h_{i+\frac12} v_{i+\frac12}^n)\right) \frac{u_i^n - u_{i+1}^n}{h_{i+\frac12}}, 
		\end{equation}
		The velocities $v_{i+\frac12}^n$ are uniformly bounded owing to~\eqref{eq:bound.v}. 
		As highlighted in the proof of Proposition~\ref{prop:estim_u_l2}, this yields a uniform bound on the quantity 
		\[\frac1{2L^n} \left( B(- L^n h_{i+\frac12} v_{i+\frac12}^n) + B(L^n h_{i+\frac12} v_{i+\frac12}^n)\right)\] as well.
		Using the elementary inequality $(p+q)^2 \leq 2 (p^2 + q^2)$, we therefore obtain that 
		\[
		\sum_{i=0}^I h_{i+\frac12} \left|\Fluxnum{i+} \right|^2 \leq \cter{cte:fluxL2.1} + \cter{cte:fluxL2.2} \|w_h^n\|_1^2
		\]
		for some quantities $\ctel{cte:fluxL2.1}$ and $\ctel{cte:fluxL2.2}$ not depending on the discretization parameters, but only on the data prescribed in the statement of Proposition~\ref{prop:estim_du_hs}. Summing the above estimate over $n \in \{1,\dots, N\}$ and using Proposition~\ref{prop:estim_u_l2} gives exactly~\eqref{eq:FL2}.
		
		Summing up the above calculations, we end up with the existence of quantities $\ctel{cte:En2.1}$ and $\ctel{cte:En2.2}$ depending only on the 
		prescribed data such that
		\(
		{|E^n|}^2 \leq \cter{cte:En2.1} + \cter{cte:En2.2}  \|w_h^n\|_1^2
		\)
		whatever the choice of $\eta_h \in S_\Mm$ provided that $\|\eta_h\|_1 \leq 1$. Therefore, this implies that 
		\[
		{\left\|\frac{w_h^n - w_h^{n-1}}{\dt}\right\|}_{-1}^2 \leq \cter{cte:En2.1} + \cter{cte:En2.2}  \|w_h^n\|_1^2. 
		\]
		Multiplying by $\dt$, summing over $n$, and using Proposition~\ref{prop:estim_u_l2} 
		proves~\eqref{eq:L2H-1}. 
	\end{proof}
	\subsection{Compactness results}
	
	In this part is presented the main compactness result for the solution to the numerical scheme \eqref{eq:scheme}.
	Before focusing on the density profile, let us start by the interfaces $\bX_{h,\dt} = (X_{0,h,\dt}, X_{1,h,\dt})$ and their 
	approximate derivatives 
	\[
	\p_{t, \dt} X_{\ell,h,\dt}(t) = \begin{cases}
\frac{X_{\ell,h,\dt}(t) - X_{\ell,h,\dt}(t-\dt)}{\dt} & \text{for}\; t \in (\dt, T], \\
\frac{X_\ell^1 - X_\ell^0}{\dt} & \text{for}\; t \in [0,\dt], 
\end{cases}
\qquad \ell \in \{0,1\}.
	\]
	
	\begin{lemma} \label{lem:compact.X}
	Under the assumptions of Proposition~\ref{prop:estim_u_l2}, there exist some Lipschitz continuous functions $X_0$ and $X_1$ defined on $[0,T]$, 
	with $X_0(0) = 0$, $X_1(0) = L^0$ and with $L(t) = X_1(t) - X_0(t) \in [L_\flat, L^\sharp]$ for all $t \in [0,T]$ 
	such that, up to a subsequence
	\begin{equation}\label{eq:X.Lip}
	X_{\ell, h, \dt} \underset{h,\dt \to 0}\longrightarrow X_\ell \quad \text{uniformly on $[0,T]$}, \qquad \ell \in \{0,1\}, 
	\end{equation}
	while 
	\begin{equation}\label{eq:X.Lip.2}
	\p_{t, \dt} X_{\ell,h,\dt}  \underset{h,\dt \to 0}\longrightarrow X_\ell' \quad \text{in the $L^\infty(0,T)$-weak-$\star$ sense}.
	\end{equation}
	\end{lemma}
	\begin{proof}
	The uniform bounds on $\Delta X_\ell^n$ yield uniform bounds on $\p_{t, \dt} X_{\ell,h,\dt}$, thus there exist $g_\ell$ in $L^\infty(0,T)$, $\ell \in \{0,1\}$, 
	such that 
	\[
	\p_{t, \dt} X_{\ell,h,\dt}  \underset{h,\dt \to 0}\longrightarrow g_\ell \quad \text{in the $L^\infty(0,T)$-weak-$\star$ sense}.
	\]
	For $\tau>0$ and $t \in [0,T-\tau)$, one furthermore has
	\[
	|X_{\ell,h,\dt}(t+\tau) - X_{\ell, h, \dt}| \leq C(\tau+\dt)
	\]
	with $C$ being a uniform upper bound for $\Delta X^n_\ell$, $n \geq 0$. 
	We can then apply the refined version of the Ascoli-Arzelà theorem \cite[Proposition 3.3.1]{AGS08} to infer the desired uniform convergence. 
	The preservation of the bounds on $L$ follows. 
	Showing that $g_\ell = X_\ell'$ in the sense of distributions in then classical. 
	\end{proof}
	
Let us now consider the compactness of the approximate density profiles on the fixed domain $w_{h,\Delta t}$.
	
		\begin{lemma} \label{lem:aubsim}
			Under the assumptions of Proposition~\ref{prop:estim_u_l2}, there exists a function $w \in L^\infty((0,T)\times(0,1)) \cap L^2(0,T);H^1(0,1))$ such that, up to a subsequence, $w_{h,\dt} \longrightarrow w$ (strongly) in $L^2((0,T)\times(0,1))$ as $h,\dt \rightarrow 0$ and $\partial_{\xi,h} w_{h,\dt} \longrightarrow \partial_\xi w$ weakly in $L^2((0,T)\times(0,1))$. Moreover, we have the strong convergence of the traces 
			\begin{equation}\label{eq:conv.traces}
			w_{h,\dt}(t, 0) \underset{h,\dt \to 0} \longrightarrow w(t,0) \quad \text{and} \quad  w_{h,\dt}(t, 1) \underset{h,\dt \to 0} \longrightarrow w(t,1) \quad 
			\text{a.e. in } (0,T), 
			\end{equation}
			and initial profile
			\begin{equation}\label{eq:conv.init}
			w_{h,\dt}(0, \xi) \underset{h \to 0} \longrightarrow w^\text{init}(\xi) = u^\text{init}(\xi L^0)\quad
			\text{for a.e.}\; \xi \in (0,1). 
			\end{equation}
		\end{lemma}

	\begin{proof}
	As $m \leq w_{h,\dt} \leq M$ a.e. in $(0,T) \times (0,1)$, there exists $w$ measurable with $m \leq w \leq M$ such that $w_{h,\dt}$ converges, up to a subsequence, 
	towards $w$ in the $L^\infty(0,T) \times (0,1)$ weak-$\star$ sense. 	
		With Propositions \ref{prop:estim_u_l2} and \ref{prop:estim_du_hs} at hand, one can directly apply the discrete Aubin-Simon lemma \cite[Theorem 3.4]{GallLatch}, and get that, up to the extraction of another subsequence, $ w_{h,\dt} \longrightarrow w$ in $L^2((0,T) ; L^2(0,1))$ as $h,\dt \rightarrow 0$. 
		Moreover, we deduce from Proposition~\ref{prop:estim_u_l2} that $\partial_{\xi,h} w_{h,\dt}$ is bounded in $L^2((0,T)\times(0,1))$, hence its weak in $L^2$ convergence towards some limit $f$.
		A straightforward adaptation in \cite[Lemma 8.2]{EGH00} shows that the limiting profile $w$ belongs to $L^2((0,T);H^1(0,1))$. The identification $f = \partial_\xi w$ is classical. 
		
	Finally, the convergence of the traces~\eqref{eq:conv.traces} can be proven by adapting the strategy of~\cite[Section 4.2]{BrenCancHilh} to our simpler one-dimensional setting, while the convergence property~\eqref{eq:conv.init} is nothing but the classical consistence result for the approximation~\eqref{eq:u0.disc}. 
	\end{proof}

\subsection{Identification of the limit as a weak solution}

The goal of this section is to establish our last theoretical result, which by the way grants the existence of a weak solution up to a minimal time thanks to Proposition~\ref{prop:Ln_bd}. 
As a preliminary, let us show some compactness and consistency result on the approximate fluxes following ideas from \cite[Lemma 4.4]{CHD11}.
\begin{lemma}\label{lem:conv.F}
Under the assumptions of Proposition~\ref{prop:estim_u_l2}, $\Ff_{h,\dt}$ defined by~\eqref{eq:Fhdt} converges, up to a subsequence, weakly in $L^2((0,T)\times(0,1))$ towards 
\begin{equation}\label{eq:def.v}
\Ff = v\,w - \frac1L \partial_\xi w \quad \text{with} \quad v(t,\xi) = (1-R) X_1'(t) - L'(t) \xi - X_0'(t).
\end{equation}
\end{lemma}
\begin{proof}
As a consequence of Lemma~\ref{lem:compact.X}, the discrete velocity field 
\[
v_{h,\dt}(t,\xi) = v_{i+\frac12}^n \quad \text{if}\; (t,\xi) \in (t^{n-1}, t^n] \times (\xi_i,\xi_{i+1})
\]
converges in the $L^\infty((0,T) \times (0,1))$ weak-$\star$ sense towards $v$ defined in~\eqref{eq:def.v}.

Starting from expression~\eqref{eq:flux.conv-diff} for the flux $\Ff_{i+\frac12}^n$ and using $B(0) = 1$, one gets that 
\begin{equation}\label{eq:Fhdt.conv-diff}
\Ff_{h,\dt} = v_{h,\dt} \overline  w_{h,\dt} - \frac1{L_{h,\dt}} \p_{\xi,h} w_{h,\dt} + r_{h,\dt}, 
\end{equation}
where we have set 
\[
r_{h,\dt}(t,\xi) =  \frac{u_i^n - u_{i+1}^n}{2 L^n h_{i+\frac12}} \left( B(- L^n h_{i+\frac12} v_{i+\frac12}^n) + B(L^n h_{i+\frac12} v_{i+\frac12}^n) - 2 B(0) \right) 
\]
and $\overline  w_{h,\dt}(t,\xi) =  \frac{u_i^n + u_{i+1}^n}2$ for $(t,\xi) \in (t^{n-1}, t^n] \times (\xi_i, \xi_{i+1})$.
Since $B$ is $1$-Lipschitz, we obtain that 
\[
|r_{h,\dt}| \leq h\, |v_{h,\dt}| \, |\p_{\xi,h} w_{h,\dt}| \underset{h,\dt\to 0}\longrightarrow 0 \quad \text{in } L^2((0,T) \times (0,1))
\]
thanks to the uniform boundedness of $v_{h,\dt}$ in $L^\infty(0,T) \times (0,1))$ and of $\p_{\xi,h} w_{h,\dt}$ in $L^2((0,T) \times (0,1))$, cf. Proposition~\ref{prop:estim_u_l2}. Moreover, one readily checks that 
\[
| \overline  w_{h,\dt}(t,\xi)  - w_{h,\dt}(t,\xi) | \leq h  |\p_{\xi,h} w_{h,\dt}|, 
\]
so that, thanks to Lemma~\ref{lem:aubsim}, $\overline  w_{h,\dt}(t,\xi)$ also converges strongly in $L^2$ towards $w$. So, using Lemmas~\ref{lem:compact.X} ad \ref{lem:aubsim}, we can pass to the limit $h,\dt \to 0$ in \eqref{eq:Fhdt.conv-diff} to recover~\eqref{eq:def.v}.
\end{proof}

We are now in position to prove the main result of this section. 
\begin{theorem}\label{th:conv}
Let $\bX = (X_0,X_1)$ be as in Lemma~\ref{lem:compact.X}, let $w$ be defined as in Lemma~\ref{lem:aubsim} and $u$ be defined by $u(t,x)= w(t,\dfrac{x-X_0(t)}{X_1(t)-X_0(t)})$ , then $(u,X_0,X_1)$ is a weak solution in the sense of Definition~\ref{def:weak}.
\end{theorem}
\begin{proof}
Relations \eqref{eq:X0}--\eqref{eq:X1} for the motion of the free boundaries rewrite 
\begin{align*}
\p_{t,\dt} X_{1,h,\dt} =& -\alpha_1 +\beta_1 w_{h,\dt}(\cdot, 1), \\
\p_{t,\dt} X_{0,h,\dt}  = &(1-R) \p_{t,\dt} X_{1,h,\dt} + \alpha_0 -  \beta_0 w_{h,\dt}(\cdot, 0).
\end{align*}
We can pass to the limit $h,\dt \to 0$ in the above relations thanks to Lemmas~\ref{lem:compact.X} and \ref{lem:aubsim}, leading to 
\(
X_1' = -\alpha_1 +\beta_1 w(\cdot, 1) \mbox{ and } 
X_0'  = (1-R) X_1' + \alpha_0 -  \beta_0 w (\cdot, 0).
\)

Now let $\varphi \in C^\infty_c([0,T)\times[0,1])$. For $i \in \{0,\dots, I+1\}$ and $n \in \{0,\dots, N\}$, we denote by $\varphi_i^n = \varphi(t^n, \xi_i)$,
by $\varphi_{h,\dt} = \Pi_{\Mm,\dt}\left( (\varphi_i^n)_{i,n} \right)$, and by $\partial_{t,\dt} \varphi_{h,\dt}$ and $\partial_{\xi,h} \varphi_{h,\dt}$ the approximate time and space derivatives of 
$\varphi_{h,\dt}$ respectively defined with formulas~\eqref{eq:dtwhdt} and \eqref{eq:dxiwhdt} with $\varphi$ instead of $w$. Note that $\varphi_i^N = 0$. 
Because of the regularity of $\varphi$, one readily shows that 
\begin{align}
\varphi_{h,\dt} \underset{h,\dt \to 0} \longrightarrow & \; \varphi \quad \text{uniformly on }[0,T] \times [0,1], 
\label{eq:conv.phi}\\
\partial_{t,\dt} \varphi_{h,\dt} \underset{h,\dt \to 0} \longrightarrow & \; \partial_t \varphi \quad \text{in the } L^\infty((0,T) \times (0,1))\text{-weak-$\star$ sense}, 
\label{eq:conv.dtphi}\\
\partial_{\xi,h} \varphi_{h,\dt} \underset{h,\dt \to 0} \longrightarrow & \; \partial_\xi \varphi \quad \text{uniformly on }[0,T] \times [0,1].
\label{eq:conv.dxiphi}
\end{align}

Multiplying the scheme~\eqref{eq:scheme_u} by $\varphi_i^{n-1}$ 
and summing over $1 \leq i \leq I$ and $1 \leq n \leq N$ yields 
\begin{equation}\label{eq:A+Bhdt}
A_{h,\dt} + B_{h,\dt}  = 0, 
\end{equation}
where we have set 
\begin{align*}
    A_{h,\dt} =\;& \sum_{n=1}^N \sum_{i=1}^n (L^n u_i^n - L^{n-1}u_i^{n-1})\varphi_i^{n-1}h_i, \\
    B_{h,\dt} = \;& \sum_{n=1}^N \dt \sum_{i=1}^n \left( \Fluxnum{i+} - \Fluxnum{i-} \right)  \varphi_i^{n-1}.
\end{align*}
Applying a summation by parts to $A_{h,\dt}$ gives 
\[
A_{h,\dt} = A_{h,\dt}' + A_{h,\dt}'' = \sum_{n=1}^N \sum_{i=1}^n L^n u_i^n (\varphi_i^{n-1} - \varphi_i^n) h_i - \sum_{i=1}^n L^0 u_i^0 \varphi_i^0.
\]
Because of the uniform convergence of $L_{h,\dt}$ towards $L$, {\em cf.} Lemma~\ref{lem:compact.X}, of the strong in $L^2$ convergence 
of $w_{h,\dt}$ towards $w$, {\em cf.} Lemma~\ref{lem:aubsim}, and thanks to~\eqref{eq:conv.dtphi}, one can pass to the limit in the first term in the above right-hand side  to obtain 
\[
A_{h,\dt}' =  \int_0^T \int_0^1 L_{h,\dt} w_{h,\dt} \p_{t,\dt} \varphi_{h,\dt}  \underset{h,\dt \to 0} \longrightarrow  \int_0^T \int_0^1 L\, w\, \p_t\varphi. 
\]
For the second contribution, one deduces from~\eqref{eq:conv.init} and \eqref{eq:conv.phi} that 
\[
A_{h,\dt}'' =  \int_0^1 L^0 w_h^0 \varphi_{h,\dt}(0,\cdot)   \underset{h \to 0}\longrightarrow  \int_0^1 L^0 u^\text{init}(L^0 \xi)  \varphi(0,\xi) \dd \xi.
\]

On the other hand, yet another summation by parts provides 
\[
 B_{h,\dt} = B'_{h,\dt} +  B''_{h,\dt} = \sum_{n=1}^N \dt \sum_{i=1}^n \Fluxnum{i+} \left(\varphi_i^{n-1} - \varphi_{i+1}^n\right) - \sum_{n=1}^N \Ff_{\frac12}^n \varphi_0^n.
\]
The first term in the above right-hand side rewrites 
\[
B'_{h,\dt} = - \int_0^T \int_0^1 \Ff_{h,\dt} \p_{\xi,h} \varphi_{h,\dt} 
\underset{h,\dt \to 0}\longrightarrow   \int_0^T \int_0^1\left (\frac1L\p_\xi w - v\, w\right) \p_\xi \varphi
\]
owing to Lemma~\ref{lem:conv.F} and to \eqref{eq:conv.dxiphi}. For the second term, one makes use of the definition~\eqref{eq:left_cond_flux} of the flux $\Ff_{\frac12}^n$ 
to write 
$$
B''_{h,\dt}  = \int_0^T (b\, w_{h,\dt}(t,0) - a) \varphi_{h,\dt}(t,0) \dd t \underset{h,\dt \to 0}\longrightarrow  \int_0^T (b\, w(t,0) - a) \varphi(t,0)\dd t
$$
thanks to  \eqref{eq:conv.phi}, to \eqref{eq:conv.traces}.
We then deduce from~\eqref{eq:A+Bhdt} and from the above calculations that $(w,X_0,X_1)$ satisfies \eqref{eq:weak-w} and all the points of Remark \ref{rk:def-weak-sol} are satisfied. It concludes the proof of Theorem~\ref{th:conv}.
\end{proof}	

\section{Numerical results} \label{sec:simu}

In this final section, we present some numerical experiments. The scheme being implicit, its solutions are obtained thanks to Newton's method, with a $10^{-10}$ tolerance between two consecutive iterations, on a uniform mesh.
The different test cases, corresponding to different parameters, are presented in Table \ref{table:test_param}. 
 The numerical results are obtained with $I = 100$ cells, and the time step is $\dt = 10^{-2}$, until different final times depending on the test case. 

Test case 1 satisfies the condition \eqref{hyp}, and one observes on Figure~\ref{fig:profiles} the convergence of the approximate solution $u_{h,\dt}$ towards the travelling wave profile $\hat u$. Condition \eqref{hyp} is satisfied neither for test case 2, for which the oxide layer disappears in finite time, nor for test case 3 where the oxide layer grows indefinitely.

\begin{table}[htb]
\centering
\caption{Parameters of the test cases.}

\begin{tabular}{ |c|c|c|c|c|c|c|c|c|c|c|c| } 
	\hline
	  & $a$ & $b$ & $\alpha_0$ & $\beta_0$ & $\alpha_1$ & $\beta_1$ & $R$ & $u^\text{init}$ & $L^0$ & $T$ & $\eqref{hyp}$ \\
	\hline
	Test case 1 & $1$ & $1$ & $1.5$ & $1$ & $0.5$ & $4$ & $2$ & $\frac{a}{b} e^{-R\hat{c}x} + 2$ & $1$ &$20$  & \ding{51} \\
	\hline
	Test case 2 & $1.75$ & $1$ & $5$ & $2$ & $5$ & $2$ & $2$ & $\frac{a}{b} e^{-R\hat{c}x} + 2$ & $1$ & $3.5$ & \ding{55}  \\
	\hline
	Test case 3 & $1$ & $1$ & $4$ & $1$ & $3$ & $1.5$ & $2$ & $\frac{a}{b} e^{-R\hat{c}x} + 2$ & $1$ & $10$ & \ding{55} \\
	\hline
\end{tabular}
\label{table:test_param}

\end{table}

\begin{figure}[htb]
\begin{minipage}{0.32\textwidth}
	\includegraphics[height=3cm]{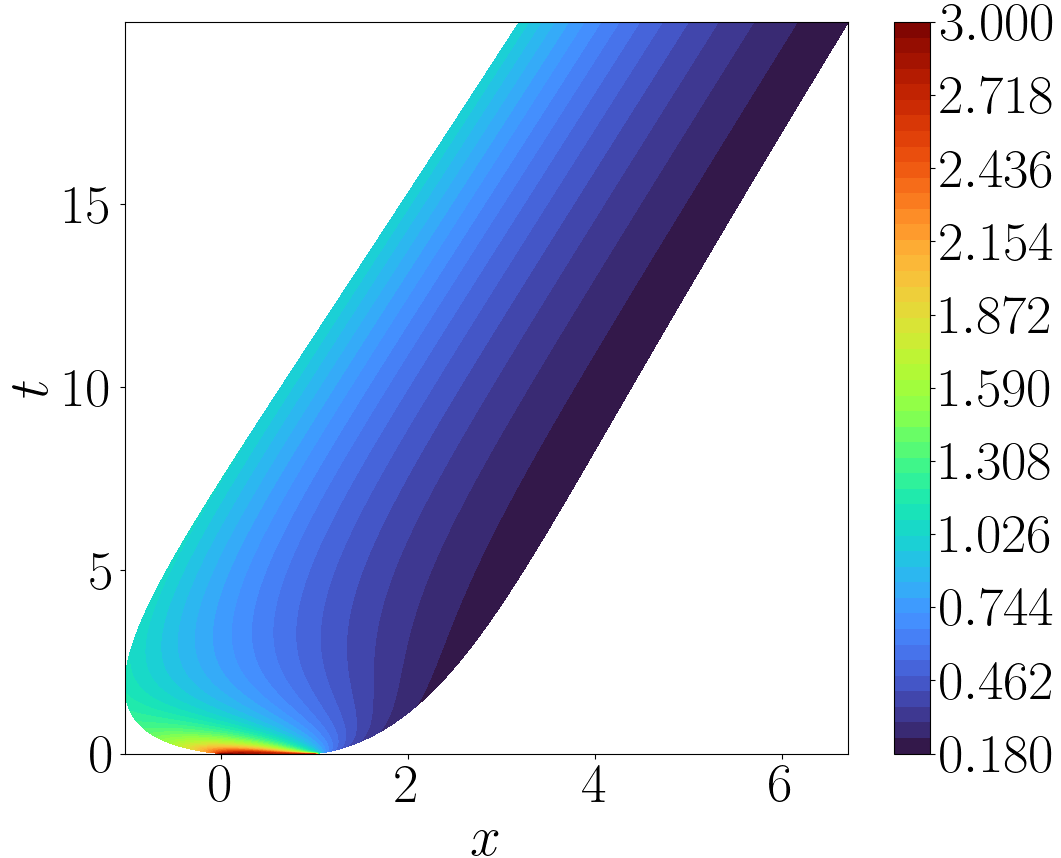}
\end{minipage}
\begin{minipage}{0.2cm}
\end{minipage}
\begin{minipage}{0.31\textwidth}
	\includegraphics[height=3cm]{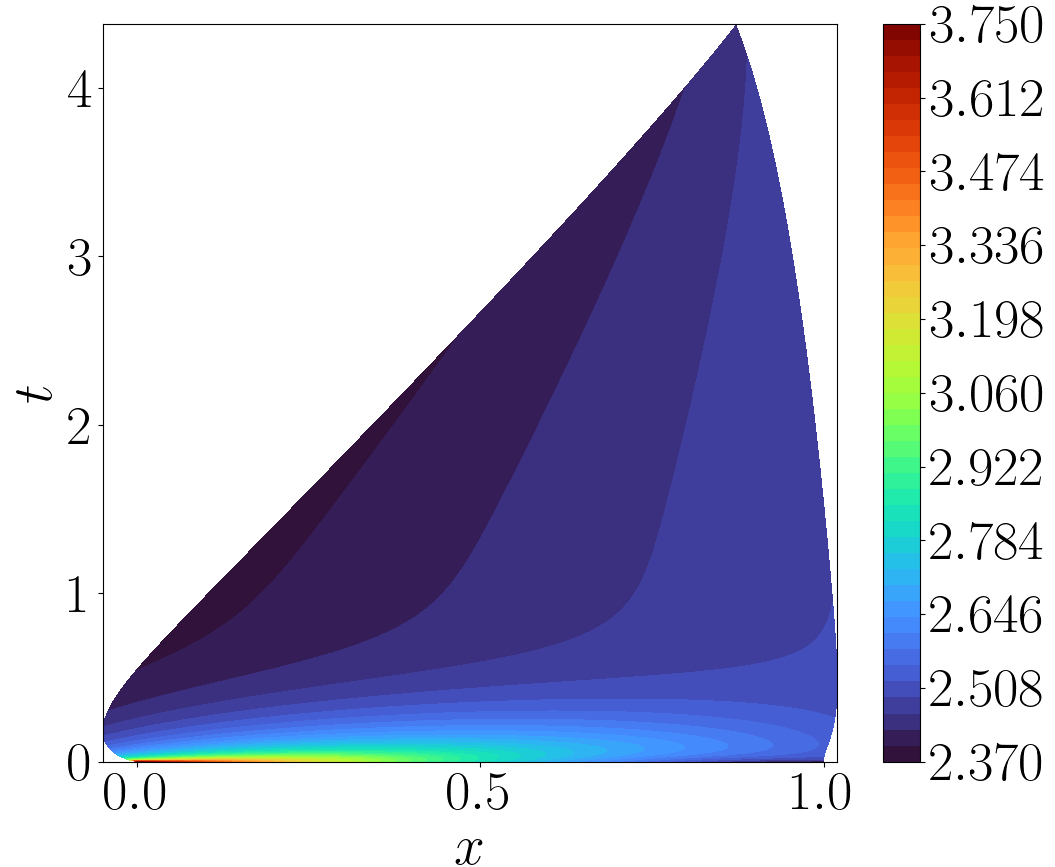}
\end{minipage}
\begin{minipage}{0.32\textwidth}
	\includegraphics[height=3cm]{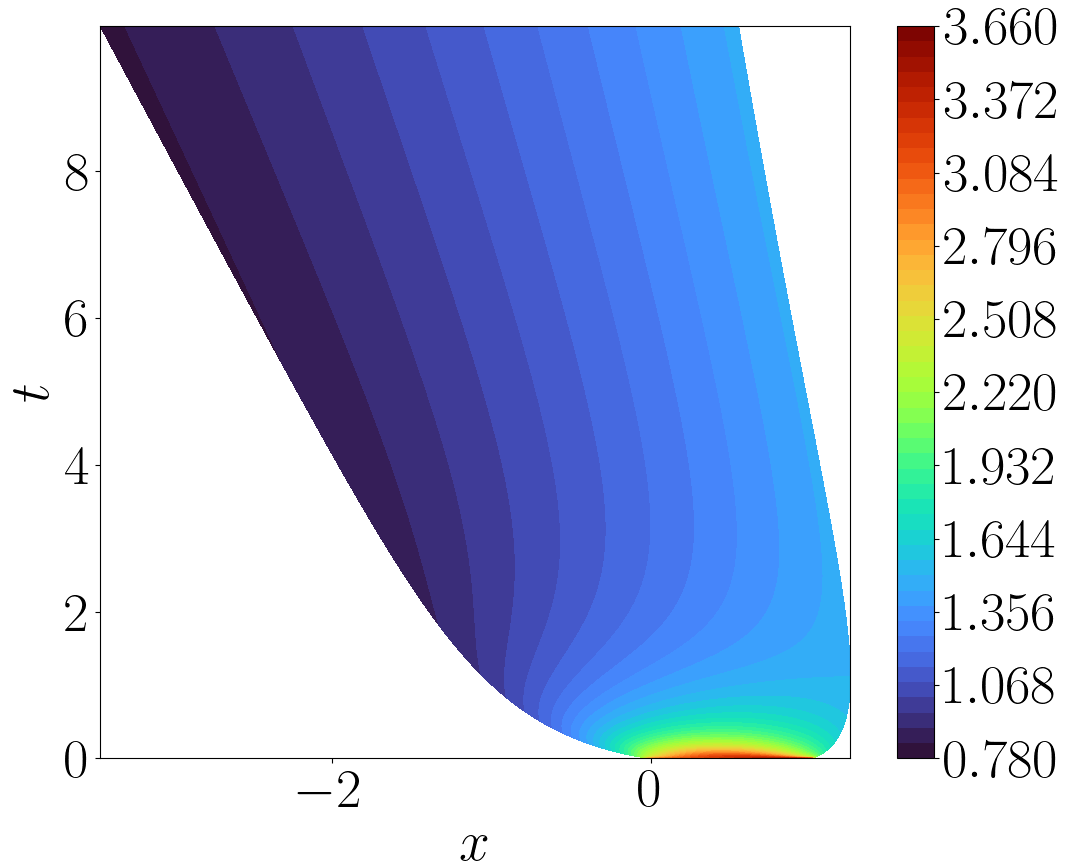}
\end{minipage}
\caption{Profiles of the approximate solution for test cases 1 (left), 2 (center) and 3 (right).}
\label{fig:profiles}
\end{figure}

\noindent
In order to compare the travelling wave profile and the other profile from test case1, both are rescaled in $[0,1]$, so it is possible compute a discrete distance between the approximate solution $u$ and the travelling wave profile $\hat{u}$. Precisely, after the rescaling, the distance between the two profiles at each time step $n$ is measured through the quantity
\(
	d^n = L^n \sum_{i=1}^{I}h_i (u_i^n - \hat{u}(\xi_i))^2.
\)
The evolution of this distance along time is plotted at the left of Figure \ref{fig:dist_log} in a semilogarithmic scale, to exhibit the exponential convergence of the (approximate) solution towards a travelling wave as time goes to $+\infty$. The evolution of the distance between the approximate width of the oxide $L^n$ and the one of the travelling wave $\hat{L}$ is also plotted, at the right of Figure \ref{fig:dist_log}. This behavior is in good accordance with the results obtained in \cite{Bataillon_elec} and \cite{Calipso} for more complex corrosion models. 
Observe on Figure \ref{fig:dist_log} that the computation ultimately reaches machine error. This is a consequence of the exactness of our scheme for the travelling wave. 
\begin{figure}[htb]
    \begin{minipage}{0.48\textwidth}
      \includegraphics[scale=0.15]{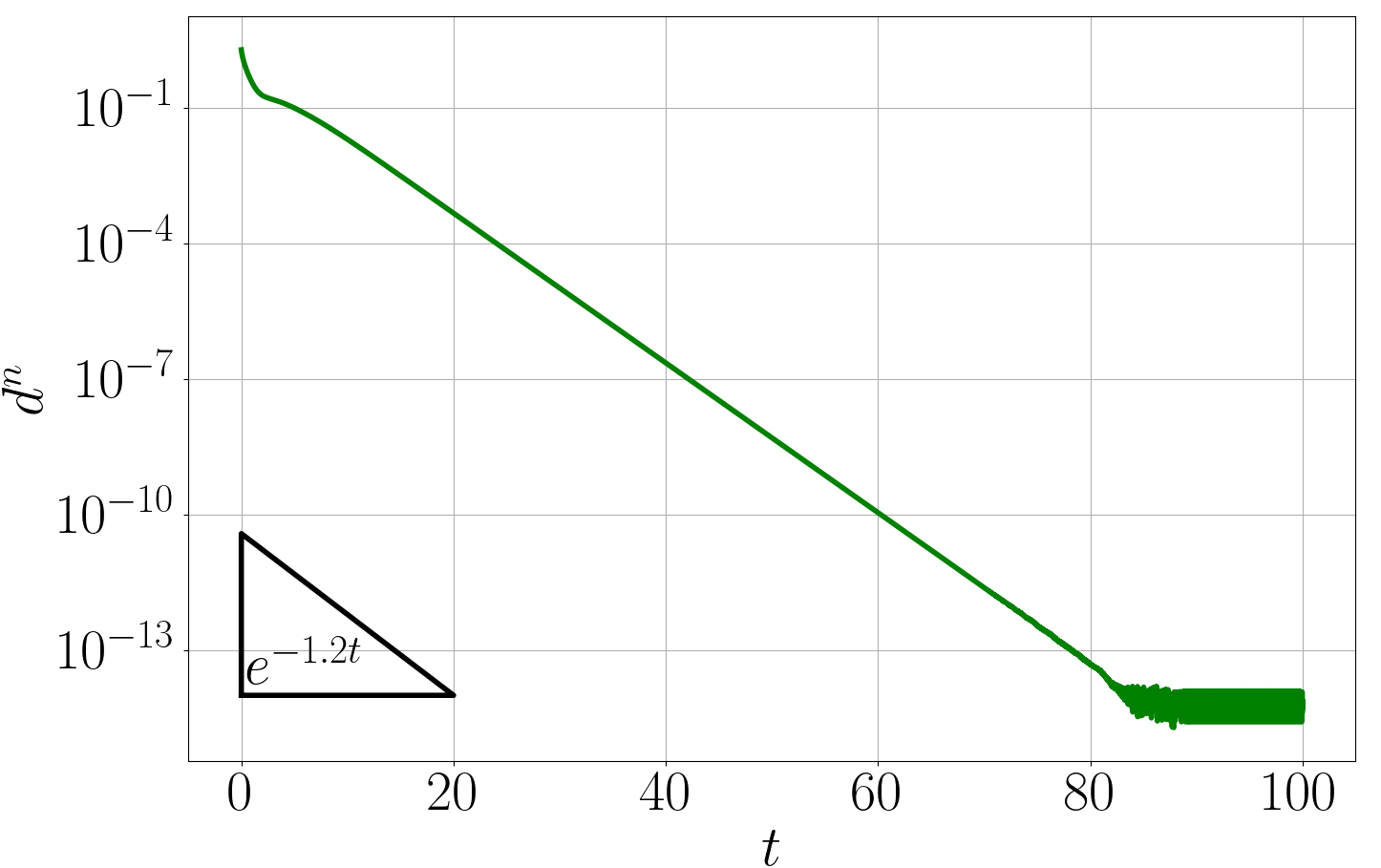}
    \end{minipage}
    \begin{minipage}{0.48\textwidth}
      \includegraphics[scale=0.15]{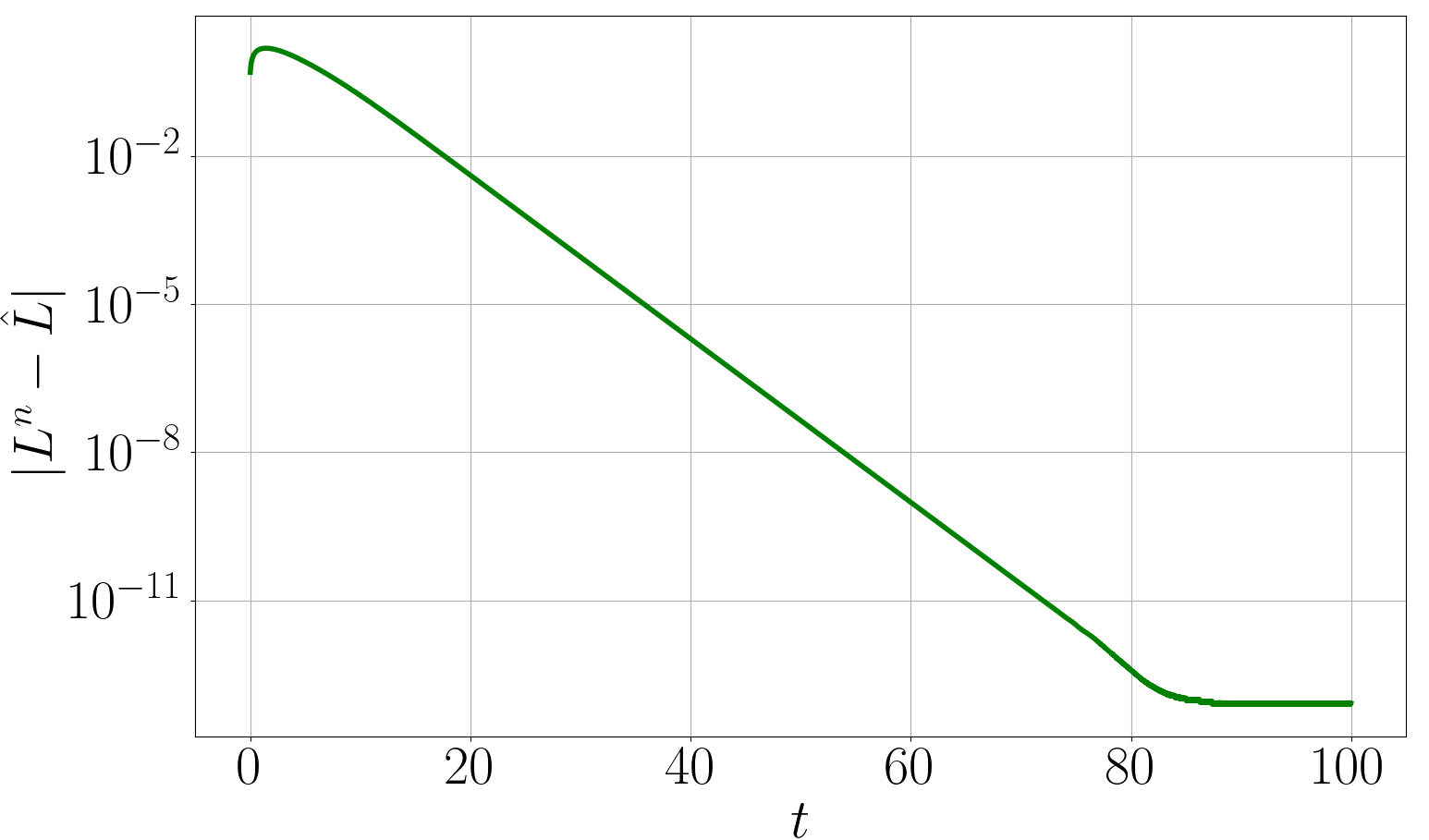}
    \end{minipage}
    \caption{Convergence towards a travelling wave profile of the concentration $u$ (left) and the width of the oxide $L$ (right).}
    \label{fig:dist_log}
\end{figure}

Finally, we compute approximate solutions $(w_{h_k, \dt_k}, X_{0,h_k, \dt_k}, X_{1,h_k, \dt_k})$ on different meshes of $I_k = 50 \times 2^k$ cells and $N_k = 10 \times 4^k$ time steps for $k \in \{0,\dots, 5\}$ for test-case 1. As we already know that the scheme is exact for the travelling wave, which has been shown on Figure~\ref{fig:dist_log} to be the long-time limit the the scheme, we consider a rather small final time $T = 0.2$ to focus on the transient regime. 
The reference solution $(w_\text{ref}, X_{0,\text{ref}}, X_{1,\text{ref}})$ is the one on the finest mesh, i.e. for $k=5$.
Introducing the orthogonal projection  $P_{\Mm_k, \dt_k}: L^2((0,T) \times (0,1)) \to S_{\Mm_k, \dt_k}$, then the error and the associated convergence rate for the profile $w_{h_k, \dt_k}$ is estimated by 
\begin{align*}
\text{err}_{w,k} =\;& \left\|w_{h_k, \dt_k} - P_{\Mm_k, \dt_k} w_\text{ref}\right\|_{L^2((0,T)\times (0,1))},
\\
\text{rate}_{w,k} =\;& \frac{\log(\text{err}_{w,k}) - \log(\text{err}_{w,k-1})}{\log h_k - \log h_{k-1}}.
\end{align*}
Similarly, denoting by $P_{\dt_k}$ the orthogonal projection of $L^2(0,T)$ on the set of piecewise constant functions on time intervals of constant length $\dt_k$, 
we denote by 
\[\text{err}_{\ell,k} = \left\|X_{\ell, h_k, \dt_k} - P_{\dt_k} X_{\ell,\text{ref}} \right\|_{L^\infty(0,T)}, 
\qquad \ell \in \{0,1\}.\]
As the interface location are functions of time only, we choose to compute the order w.r.t. the time discretization, i.e.,  
\[
\text{rate}_{\ell,k} = \frac{\log(\text{err}_{\ell,k}) - \log(\text{err}_{\ell,k-1})}{\log \dt_k - \log \dt_{k-1}}, 
\qquad \ell \in \{0,1\}, 
\]
which is in our case twice smaller than the one w.r.t. the space discretization.
One observes in Table~\ref{tab:conv} that the scheme is second order accurate in space and 
first order accurate in time, as expected from the choice of Scharfetter-Gummel fluxes and of the 
backward Euler time integration scheme. 
\begin{table}[htb]
\caption{Convergence of the scheme under grid refinement}
\centering
\resizebox{\textwidth}{!}{
\begin{tabular}{|c|c|c||c|c||c|c||c|c|}
\hline
$k$ & $h_k$ & $\dt_k$ & $\text{err}_{w,k}$ & $\text{rate}_{w,k}$ & $\text{err}_{0,k}$ 
& $\text{rate}_{0,k}$ & $\text{err}_{1,k}$ & $\text{rate}_{1,k}$ \\
\hline 
0 & $ {2{.}0e{-2}}$ & $ {2{.}0e{-2}}$ & $ {3{.}36e{-3}}$ & -- & $ {7{.}85e{-3}}$ & -- & $ {5{.}40e{-3}}$ & --  \\
1 & $ {1{.}0e{-2}}$ & $ {5{.}0e{-3}}$ & $ {8{.}90e{-4}}$ & 1{.}91 & $ {2{.}22e{-3}}$ & 0{.}90 & $ {1{.}86e{-3}}$ & 0{.}76 \\
2 & $ {5{.}0e{-3}}$ & $ {1{.}25e{-3}}$ & $ {1{.}96e{-4}}$ & 2{.}17 & $ {6{.}17e{-4}}$ & 0{.}92 & $ {5{.}4e{-4}}$ & 0{.}89 \\
3 & $ {2{.}5e{-3}}$ & $ {3{.}125e{-4}}$ & $ {3{.}77e{-5}}$ & 2{.}38 & $ {1{.}56e{-4}}$ & 0{.}99 & $ {1{.}4e{-4}}$ & 0{.}97 \\
4 & $ {1{.}25e{-3}}$ & $ {7{.}81e{-5}}$ & $ {5{.}71e{-6}}$ & 2{.}72 & $ {3{.}18e{-5}}$ & 1{.}14 & $ {2{.}89e{-5}}$ & 1{.}14 \\
\hline 
\end{tabular}}
\label{tab:conv}
\end{table}

\section{Conclusion and prospects}

The work conducted in this paper on a simplified model for corrosion, with moving boundaries, presents both theoretical and numerical results providing rigorous foundations in this simplified setting to phenomena observed numerically on more complex models, especially concerning the crucial role played by travelling waves when such profiles exist. 

The numerical strategy employed in the code CALIPSO~\cite{Calipso}, which also enters the framework of ALE finite volumes, and which also makes use of Scharfetter-Gummel fluxes, is by the way validated. 
Indeed, the existence of a travelling wave profile as presented in Proposition \ref{th:trav_wave}, under some explicit assumption on the kinetic parameters, joins the work that is done in \cite{Breden} and \cite{ChainGall}. Moreover, in this one species case, the expression of the travelling wave is explicit, and this motivates the choice of Scharfetter-Gummel fluxes for the discretization employed in~\cite{Calipso}. This kind of fluxes guarantees the preservation of the travelling wave profile by the numerical scheme, and is compatible with the rich structure of the continuous problem, which admits an infinite number of free energies. 

Beyong the expected convergence orders --$2$ in space and $1$ in time--, the numerical results show the convergence in the long-time regime towards the unique travelling wave. The theoretical analysis of this asymptotic behavior will be conducted in future works. 
We will also extend our results  to richer models for corrosion inspired from~\cite{Bataillon_elec}, including multiple species and self-consistent electric drift. 

\section*{Acknowledgements} The authors acknowledge partial support from the  European Union’s Horizon 2020 research and
innovation programme under grant agreement No 847593 (WP DONUT) and from Labex CEMPI (ANR-11-LABX-0007-01).

\bibliographystyle{plain}
\bibliography{bibliographie}

\begin{thebibliography}{10}

\bibitem{AM_arXiv}
C.~Alamichel and N.~Meunier.
\newblock Existence of traveling wave for a coupled incompressible {D}arcy's
  free boundary model with undercooling effect and surface tension.
\newblock arXiv:2501.04576, 2025.

\bibitem{AGS08}
L.~Ambrosio, N.~Gigli, and G.~Savar{\'e}.
\newblock {\em Gradient flows in metric spaces and in the space of probability
  measures}.
\newblock Lectures in Mathematics ETH Z\"urich. Birkh\"auser Verlag, Basel,
  second edition, 2008.

\bibitem{AMTU01}
A.~Arnold, P.~Markowich, G.~Toscani, and A.~Unterreiter.
\newblock On convex {S}obolev inequalities and the rate of convergence to
  equilibrium for {F}okker-{P}lanck type equations.
\newblock {\em Comm. Partial Differential Equations}, 26(1-2):43--100, 2001.

\bibitem{Bataillon_elec}
C.~Bataillon, F.~Bouchon, C.~Chainais-Hillairet, C.~Desgranges, E.~Hoarau,
  F.~Martin, S.~Perrin, M.~Turpin, and J.~Talandier.
\newblock Corrosion modelling of iron based alloy in nuclear waste repository.
\newblock {\em Electrochim. Acta}, 55(15):4451--4467, 2010.

\bibitem{Calipso}
C.~Bataillon, F.~Bouchon, C.~Chainais-Hillairet, J.~Fuhrmann, E.~Hoarau, and
  R.~Touzani.
\newblock Numerical methods for the simulation of a corrosion model with moving
  oxide layer.
\newblock {\em J. Comput. Phys.}, 231(18), 2012.

\bibitem{BC14}
F.~Bolley and J.~A. Carrillo.
\newblock Nonlinear diffusion: geodesic convexity is equivalent to
  {W}asserstein contraction.
\newblock {\em Comm. Partial Differential Equations}, 39(10):1860--1869, 2014.

\bibitem{Breden}
M.~Breden, C.~Chainais-Hillairet, and A.~Zurek.
\newblock Existence of traveling wave solutions for the {D}iffusion {P}oisson
  {C}oupled {M}odel: a computer-assisted proof.
\newblock {\em ESAIM Math. Model. Numer. Anal.}, 55(4):1669--1697, 2021.

\bibitem{BrenCancHilh}
K.~Brenner, C.~Cancès, and D.~Hilhorst.
\newblock Finite volume approximation for an immiscible two-phase flow in
  porous media with discontinuous capillary pressure.
\newblock {\em Comput. Geosci.}, 17:573--597, 2013.

\bibitem{CCCE_IFB}
C.~Cancès, J.~Cauvin-Vila, C.~Chainais-Hillairet, and V.~Ehrlacher.
\newblock Cross-diffusion systems coupled via a moving interface.
\newblock {\em Interfaces Free Bound.}, 2024.
\newblock Online first.

\bibitem{Cances_ZAMP}
C.~Cancès, C.~Chainais-Hillairet, B.~Merlet, F.~Raimondi, and J.~Venel.
\newblock Mathematical analysis of a thermodynamically consistent reduced model
  for iron corrosion.
\newblock {\em Z. Angew.Math.Phys.}, 74(96), 2023.

\bibitem{CJ05}
P.~Cermelli and M.~Jabbour.
\newblock Multispecies epitaxial growth on vicinal surfaces with chemical
  reactions and diffusion.
\newblock {\em Proc. R. Soc. A}, 461(2063):3483--3504, 2005.

\bibitem{CB08}
C.~Chainais-Hillairet and C.~Bataillon.
\newblock Mathematical and numerical study of a corrosion model.
\newblock {\em Numer. Math.}, 110:1--25, 2008.

\bibitem{CHD11}
C.~Chainais-Hillairet and J.~Droniou.
\newblock Finite-volume schemes for noncoercive elliptic problems with
  {N}eumann boundary conditions.
\newblock {\em IMA J. Numer. Anal.}, 31(1):61--85, 2011.

\bibitem{ChainGall}
C.~Chainais-Hillairet and T.~O. Gallouët.
\newblock Study of a pseudo-stationary state for a corrosion model: existence
  and numerical approximation.
\newblock {\em Nonlinear Anal. Real World Appl.}, 31:38--56, 2016.

\bibitem{Chainais_dcds}
C.~Chainais-Hillairet and I.~Lacroix-Violet.
\newblock On the existence of solutions for a drift-diffusion system arising in
  corrosion modelling.
\newblock {\em DCDS-B}, 20(1):77--92, 2014.

\bibitem{CMZ18}
C.~Chainais-Hillairet, B.~Merlet, and A.~Zurek.
\newblock Convergence of a finite volume scheme for a parabolic system with a
  free boundary modeling concrete carbonation.
\newblock {\em ESAIM Math. Model. Numer. Anal.}, 52(2):457--480, 2018.

\bibitem{Chatard_FVCA6}
M.~Chatard.
\newblock Asymptotic behavior of the {S}charfetter-{G}ummel scheme for the
  drift-diffusion model.
\newblock In {\em Finite volumes for complex applications {VI}. {P}roblems \&
  perspectives. {V}olume 1, 2}, volume~4 of {\em Springer Proc. Math.}, pages
  235--243. Springer, Heidelberg, 2011.

\bibitem{Deimling}
K.~Deimling.
\newblock {\em Nonlinear Functional Analysis}.
\newblock Springer Berlin, Heidelberg, 1985.

\bibitem{ALE}
J.~Donea, A.~Huerta, J.-Ph. Ponthot, and A.~Rodríguez-Ferran.
\newblock Arbitrary {L}agrangian–{E}ulerian methods.
\newblock {\em Encyclopedia of Computational Mechanics}, 2004.

\bibitem{kangourou_2018}
J.~Droniou, R.~Eymard, T.~Gallou\"et, C.~Guichard, and R.~Herbin.
\newblock {\em The Gradient Discretisation Method}, volume~42 of {\em
  Math\'ematiques et Applications}.
\newblock Springer International Publishing, Cham, 2018.

\bibitem{EGH00}
R.~Eymard, T.~Gallou\"et, and R.~Herbin.
\newblock Finite volume methods.
\newblock Ciarlet, P. G. (ed.) et al., in Handbook of numerical analysis.
  North-Holland, Amsterdam, pp. 713--1020, 2000.

\bibitem{GajewGrog}
H.~Gajewski and K.~Gröger.
\newblock Reaction-diffusion processes of electrically charged species.
\newblock {\em Math. Nachr.}, 177:109--130, 1996.

\bibitem{GallLatch}
T.~Gallouët and J.-C. Latché.
\newblock Compactness of discrete approximate solutions to parabolic
  {PDE}s–{A}pplication to a turbulence model.
\newblock {\em Comm. Pure Appl. Anal.}, 11(6), 2012.

\bibitem{Glitz1}
A.~Glitzky.
\newblock Analysis of electronic models for solar cells including energy
  resolved defect densities.
\newblock {\em Math. Methods. App. Sci.}, 34(16):1980--1998, 2011.

\bibitem{Glitz2}
A.~Glitzky.
\newblock An electronic model for solar cells including active interfaces and
  energy resolved defect densities.
\newblock {\em SIAM Journal on Mathematical Analysis}, 44(6):3874--3900, 2012.

\bibitem{GlitzGrogHunl}
A.~Glitzky, K.~Gröger, and R.~Hünlich.
\newblock Free energy and dissipation rate for reaction diffusion processes of
  electrically charged species.
\newblock {\em Appl. Anal.}, 60:201--217, 1996.

\bibitem{Il'in69}
A.~M. Il'in.
\newblock Differencing scheme for a differential equation with a small
  parameter affecting the highest derivative.
\newblock {\em Math. Notes Acad. Sci. USSR}, 6:596–602, 1969.

\bibitem{LMV96}
R.~D. Lazarov, I.~D. Mishev, and P.~S. Vassilevski.
\newblock Finite volume methods for convection-diffusion problems.
\newblock {\em SIAM J. Numer. Anal.}, 33(1):31--55, 1996.

\bibitem{LLRV09}
B.~Li, J.~Lowengrub, A.~Ratz, and A.~Voigt.
\newblock Geometric evolution laws for thin crystalline films: modeling and
  numerics.
\newblock {\em Commun. Comput. Phys.}, 6(3):433, 2009.

\bibitem{Mie11}
A.~Mielke.
\newblock A gradient structure for reaction-diffusion systems and for
  energy-drift-diffusion systems.
\newblock {\em Nonlinearity}, 24(4):1329--1346, 2011.

\bibitem{Mielke}
A.~Mielke.
\newblock Non-equilibrium steady states as saddle points and {EDP}-convergence
  for slow-fast gradient systems.
\newblock {\em Journal of Mathematical Physics}, 64(12), 2023.

\bibitem{PV99}
A.~Pimpinelli and J.~Villain.
\newblock {\em Physics of cristal growth}.
\newblock Cambridge University Press, 1999.

\bibitem{PP10}
J.~W. Portegies and M.~A. Peletier.
\newblock Well-posedness of a parabolic moving-boundary problem in the setting
  of {W}asserstein gradient flows.
\newblock {\em Interfaces Free Bound.}, 12(2):121--150, 2010.

\bibitem{RB23}
V.~Rybalko and L.~Berlyand.
\newblock Emergence of traveling waves and their stability in a free boundary
  model of cell motility.
\newblock {\em Trans. Amer. Math. Soc.}, 376(03):1799--1844, 2023.

\bibitem{ScharfGumm}
D.L. Scharfetter and H.K. Gummel.
\newblock Large-signal analysis of a silicon {R}ead diode oscillator.
\newblock {\em IEEE Transactions on Electron Devices}, 16(1):64--77, 1969.

\end{thebibliography}

\end{document}